\documentclass[]{interact}

\usepackage{epstopdf}% To incorporate .eps illustrations using PDFLaTeX, etc.
\usepackage[caption=false]{subfig}% Support for small, `sub' figures and tables

\usepackage[numbers,sort&compress]{natbib}% Citation support using natbib.sty
\bibpunct[, ]{[}{]}{,}{n}{,}{,}% Citation support using natbib.sty
\usepackage{esdiff}

\newcommand{\evalat}[2]{\left. #1 \right\rvert_{#2}}

\graphicspath{ {./bilderSubmission/} }

\usepackage[latin1]{inputenc}
\usepackage{url}
\theoremstyle{plain}% Theorem-like structures provided by amsthm.sty
\newtheorem{theorem}{Theorem}[section]
\newtheorem{lemma}[theorem]{Lemma}

\newtheorem{proposition}[theorem]{Proposition}

\theoremstyle{definition}
\newtheorem{definition}[theorem]{Definition}

\theoremstyle{remark}

\begin{document}

%%\articletype{ARTICLE TEMPLATE}% Specify the article type or omit as appropriate

\title{On qualitative analysis of a discrete time SIR epidemical model}

\author{
\name{J. Hallberg Szabadv\'ary \textsuperscript{a}\thanks{Email: johanh@math.su.se} and Y. Zhou\textsuperscript{b}\thanks{CONTACT J. Hallberg Szabadv\'ary. Email: johanh@math.su.se}}
\affil{\textsuperscript{a} \textsuperscript{b}Department of Mathematics, Stockholm University, Sweden}
}

\maketitle

\begin{abstract}
The main purpose of this paper is to study the local dynamics and bifurcations of a discrete-time SIR epidemiological model. 
The existence and stability of disease-free and endemic fixed points are investigated along with a fairly complete classification of the systems bifurcations, in particular, a complete analysis on local stability and codimension 1 bifurcations in the parameter space. Sufficient conditions for positive trajectories are given.
The existence of a 3-cycle is shown, which implies the existence of cycles of arbitrary length by the celebrated Sharkovskii's theorem.
Generacity of some bifurcations is examined both analytically and through numerical computations. Bifurcation diagrams along with numerical simulations are presented. The system turns out to have both rich and interesting dynamics.

\end{abstract}


\section{Introduction}

In the recent two decades, there are many research papers dealing with discrete epidemic models in order to investigate the transmission dynamics of infectious diseases.  See e.g.,  \cite{badone,CCY,IzV1,IMV,JE,LMB,Sek,ZMB}.    It is believed  that they  are more appropriate approaches to understand disease transmission dynamics and to evaluate eradication policies because they permit arbitrary time-step units, preserving the basic features of corresponding continuous-time models. Furthermore, this allows better use of statistical data for numerical simulations due to the reason that the infection data are compiled at discrete given time intervals.  
In this paper we consider a discrete-time version of the SIR model in which the growth of the susceptible population, some inhibitory effects and death rates have been accounted for.  More precisely we consider the following system 
\begin{equation}
\begin{aligned}
\widetilde{S}_{n+1} &= \rho \widetilde{S}_n(1-\frac{\widetilde{S}_n}{c}) - \frac{\widetilde{\beta} \widetilde{S}_n\widetilde{I}_n}{1+\widetilde{a}\widetilde{S}_n}\\
\widetilde{I}_{n+1} &=  \frac{\widetilde{\beta} \widetilde{S}_n\widetilde{I}_n}{1+\widetilde{a}\widetilde{S}_n}+(1-\mu-\gamma)\widetilde{I}_n \\ 
\widetilde{R}_{n+1} &= \gamma \widetilde{I}_n + (1-\lambda)\widetilde{R}_n
\end{aligned}
\label{unscaled full syst}
\end{equation}
where $\widetilde{\beta} \widetilde{S}_n\widetilde{I}_n/(1+\widetilde{a}\widetilde{S}_n)$ is the force of infection, $\tilde{a}$ measures the inhibitory effect, for exampe due to public health measures imposed on the group of susceptibles, $\rho$ is the per capita growth rate for the susceptibles; individuals are born susceptible and there is no inhereted imunity. We assume $a\neq0$. Further parameters are $\gamma$, the recovery rate of the infected individuals, $\mu$ and $\lambda$ that are death rates of infected and removed respectively. Hence clearly $\mu,\gamma,\lambda<1$, and in fact since the fraction of infected that are removed due to death or recovery in each time step cannot exceed 1 we define $K=\mu+\gamma$ and assume $0<K<1$. The growth of the susceptible population is thus assumed to be \textit{logistic} which essentially means that the population grows rapidly when it is small, and more slowly as it approaches some carrying capacity, which in our case is $c$. It is important to note that this term means that the total population is not constant.

Note that $\widetilde{R}_n$ does not appear in the other two equations. It can thus be ignored on analysis of the system since it will not affect the system dynamics. Hence our main concern is the reduced model
\begin{equation}
\begin{aligned}
\widetilde{S}_{n+1} &= \rho \widetilde{S}_n(1-\frac{\widetilde{S}_n}{c}) - \frac{\widetilde{\beta} \widetilde{S}_n\widetilde{I}_n}{1+\widetilde{a}\widetilde{S}_n}\\
\widetilde{I}_{n+1} &=  \frac{\widetilde{\beta} \widetilde{S}_n\widetilde{I}_n}{1+\widetilde{a}\widetilde{S}_n}+(1-K)\widetilde{I}_n. \\ 
\end{aligned}
\label{unscaled full syst}
\end{equation}
To simplify our analysis we scale the variables  $\widetilde{S}$ and $\widetilde{I}$ by $\widetilde{S}=\alpha_1S, \widetilde{I}=\alpha_2I$ where $\alpha_1\neq0,\alpha_2\neq0$ are scaling constants to be determined. Then we have
\begin{equation}\label{scaling}
\begin{aligned}
	\alpha_1S_{n+1} &= \alpha_1S_n+\rho \alpha_1S_n(1-\frac{\alpha_1S_n}{c}) - \frac{\widetilde{\beta} \alpha_1\alpha_2S_nI_n}{1+\widetilde{a}\alpha_1S_n}\\
	\alpha_2I_{n+1} &= (1-K)\alpha_2I_n + \frac{\widetilde{\beta} \alpha_1\alpha_2S_nI_n}{1+\widetilde{a}\alpha_1S_n},
\end{aligned}
\end{equation}
equivalently,
\begin{equation}\label{scaling simplified}
\begin{aligned}
	S_{n+1} &= (1+\rho)S_n(1-\frac{\alpha_1\rho S_n}{c(1+\rho)}) - \frac{\widetilde{\beta} \alpha_2S_nI_n}{1+\widetilde{a}\alpha_1S_n}\\
	I_{n+1} &= (1-K)I_n + \frac{\widetilde{\beta} \alpha_1S_nI_n}{1+\widetilde{a}\alpha_1S_n}.
\end{aligned}
\end{equation}
Choosing $\frac{\alpha_1\rho}{c(1+\rho)}=1$ and $\widetilde{\beta}\alpha_1=\widetilde{\beta}\alpha_2$ yields $\alpha_1=\alpha_2 = \frac{c(1+\rho)}{\rho}$. Let $\beta=\alpha_1\widetilde{\beta}, a=\alpha_1\widetilde{a}$ and $r=1+\rho$ we get our equivalent system
\begin{equation}\label{syst}
\begin{aligned}
	S_{n+1} &= rS_n(1-S_n) - \frac{\beta S_nI_n}{1+aS_n}\\
	I_{n+1} &= (1-K)I_n + \frac{\beta S_nI_n}{1+aS_n}.
\end{aligned}
\end{equation}
where clearly $\beta>0, \alpha\geq0$ and $r>1$. The sytem \eqref{syst}, is the same as in the paper \cite{badone}, where the authors present some analysis and numerical simulations, indicating local stability of fixed points and bifurcation to periodic doubling but the analysis is short of rigorous, and far from complete.  This leads to an example which should indicate a limit cycle but in fact it is a case of a stable fixed point. 
Our aim in this paper lies on mathematical analysis of local stability of fixed points and other dynamical behaviors such as periodic doubling, limit cycles and their stability, and other bifurcations. The aim is to provide dynamical insight for modelers who wish to apply such models. We mention the following two arguments. First it is interesting from dynamical systems point of view, because this is a rational map, just a little more complicated than polynomial maps which often appears in population models that include competetive enviroments as discussed in \cite{modelsBiology}. Second, this system can be viewed as a discretization of a continuous model such as described in \cite{SIRLogist}.  Our analysis provides a systematic way for choices of step size, for instance using Euler's method, to avoid undesired dynamical behavior in computation because it is well-known that a discrete system exhibits  dynamical behaviours not existing in the original continuous system.

The rest of the paper is organised as follows:  We present mathematical theory which is used in our analysis, and study positive trajectories of the system \eqref{syst} in Section~2.  In Section~3 we show that there are at most two fixed points and study their local stability.      In Sections~4 and 5 we give a complete analysis on flip and Neimark-Sacker's bifurcation respectively.   We present numerical simulations in Section~6, and provide bifurcation diagrams for some typical settings of parameters as well as discussions on period \(3\). Other bifurcations  and possible chaotic behavior is supported by the computation of Lyapunov exponents.   We conclude the paper by a discussion on epidemiological relevance and  possible further investigations in Section~7. The lengthy computations are collected in the Appendix.

\section{Preliminaries}

In this section we first collect theory for analysis of dynamical system used in this study, for details we refer to \cite{kuznetsov}.  Then we show some properties of the mapping used in the model, followed by a discussion of forward positivity.

\subsection{Dynamical system preliminaries}

For simplicity we say a fixed point of the dynamical system is stable if it is asymptotically stable.   The following  local stability theorem plays the central role in stability analysis.  

\begin{theorem}\label{sufficient conditions}
Consider a discrete-time dynamical system
$$
x\mapsto f(x),~x\in\mathbb{R} ^n,
$$
where f is smooth. Suppose it has a fixed point $x^*$, so that $f(x^*)=x^*$, and denote by $A$ the Jacobian matrix of $f(x)$ evaluated at $x^*$. Then the fixed point is locally asymptotically stable if all eigenvalues $\mu_1,\mu_2,\dots,\mu_n$ of $A$ satisfy $|\mu|<1$.
\end{theorem}

For our analysis the following proposition is useful. 
\begin{proposition}
 Consider a $2\times2$-matrix $A=(a_{ij})$.  Then its characteristic polynomial 
 \[
p(\mu) =
\mu^2-\text{\em trace}(A)\mu + \det(A),
\]
has all zeros inside the unit circle if and only if 
\begin{equation}\label{stability conditions trace det}
\begin{aligned}
|\text{\em trace}(A)|&<1+\det(A)\\
|\det(A)|&<1.
\end{aligned}
\end{equation}
\end{proposition}
Hence we have found that for a fixed point $x^*$ of a two-dimensional discrete-time smooth dynamical system, with Jacobian matrix $A$ evaluated at $x^*$, sufficient conditions for stability of $x^*$ are \eqref{stability conditions trace det}.

Now consider a system that depends on parameters, which we write as
\begin{equation}\label{generic syst. param.}
x\mapsto f(x,\alpha)
\end{equation}
were $x\in\mathbb{R} ^n$ and $\alpha\in\mathbb{R} ^m$. As the parameters vary, the phase portrait also varies, and there are two possibilities. Either the system remains topologically equivalent to the original one, or its topology changes. 

\begin{definition}
The appearance of a topologically non-equivalent phase portrait under variation of parameters is called a bifurcation.
\end{definition}

Thus, a bifurcation is a change of the topological type of the system as its parameters pass through a \textit{bifurcation (critical) value}.

\begin{definition}
The codimension of a bifurcation is the difference between the dimension of the parameter space and the dimension of the corresponding bifurcation boundary. Or equivalently, the codimension is the number of independent conditions determining the bifurcation.
\end{definition}

\begin{definition}\label{def:codim1}
The following three bifurcation types are possible in codimension one:
\begin{itemize}
\item The bifurcation associated with the appearance of $\mu_1=1$ is called a {\em fold bifurcation}. \\
\item The bifurcation associated with the appearance of $\mu_1=-1$ is called a {\em flip-} or {em period-doubling bifurcation}.\\
\item The bifurcation associated with the appearance of $\mu_{1,2}=e^{\pm i\theta_0}, 0<\theta_0<\pi$ is called a {\em Neimark-Sacker bifurcation}.
\end{itemize}
\end{definition}
Note that flip and fold bifurcation may appear in one-dimensional systems, while Neimark-Sacker requires at least dimension two. 

\begin{theorem}[Generic flip]\label{generic flip}
Suppose that a one-dimensional system
$$
x\mapsto f(x,\alpha),~x\in\mathbb{R} ,~\alpha\in\mathbb{R} ,
$$
with smooth map $f$, has at $\alpha = 0$ the fixed point $x^*=0$, and let $\mu = f_x(0,0)=-1$, where $f_x$ denotes derivative. Assume that the following nondegeneracy conditions are satisfied:
\begin{equation}\tag{B.1}\label{generic flip conditions 1}
\frac{1}{2}(f_{xx}(0,0))^2+\frac{1}{3}f_{xxx}(0,0) \neq 0
\end{equation}
\begin{equation}\tag{B.2}\label{generic flip conditions 2}
f_{x\alpha}(0,0) \neq 0. 
\end{equation}
Then there are smooth invertible coordinate and parameter changes transforming the system into
$$
\eta \mapsto -(1+\beta)\eta\pm\eta^3 + O(\eta^4).
$$
\end{theorem}
The proof which is given in in Chapter 4 in \cite{kuznetsov} is not difficult but we do not give it here. The system 
\begin{equation}\label{topological normal form flip}
\eta \mapsto -(1+\beta)\eta\pm\eta^3 
\end{equation}
is called the \textit{topological normal form} for the flip bifurcation. The sign of the cubic term depends on the sign of
$$
c(0) = \frac{1}{4}(f_{xx}(0,0))^2+\frac{1}{6}f_{xxx}(0,0).
$$
Any generic, scalar, one-parameter system that satisfy the conditions in the theorem is locally topologically equivalent near the origin to \eqref{topological normal form flip}. Depending on the sign of the cubic term, the flip is called stable or unstable. If the cubic term is positive, the flip is stable, which means that the 2-cycle thus appearing is stable.

Regarding the Neimark-Sacker bifurcation we refer to \cite{kuznetsov} for the relevant theorem and normal form. We just state the nondegeneracy conditions:
\begin{equation}\tag{C.1}\label{generic NS condition 1}
\rho'(0)\neq0,
\end{equation}
\begin{equation}\tag{C.2}\label{generic NS condition 2}
e^{ik\theta_0}\neq 1~\text{for}~k=1,2,3,4,
\end{equation}
\begin{equation}\tag{C.3}\label{generic NS condition 3}
d(0)\neq0,
\end{equation}
where the system has smooth map $f(x,\alpha), ~x\in\mathbb{R} ^2$ with eigenvalues $\mu_{1,2}(\alpha) = \rho(\alpha)e^{i\varphi(\alpha)}$, where $\varphi(0)=\theta_0$. We will return to the third condition later.

Following \cite{kuznetsov} we write the system as
\begin{equation}\label{ax+fx}
\tilde{x} = Ax+F(x), x\in\mathbb{R} ^n
\end{equation}
where $F(x)=O(||x||^2)$ is a smooth function with Taylor expansion near $x^*=0$ as
\begin{equation}\label{taylor}
F(x) = \frac{1}{2}B(x,x) + \frac{1}{6}C(x,x,x) + O(||x||^4),
\end{equation}
where $B(x,y)$ and $C(x,y,z)$ are multilinear functions. In coordinates we have 
\begin{equation}\label{Bi}
B_i(x,y) = \sum_{j,k=1}^n \evalat{\diffp{F_i(\xi)}{{\xi_j}{\xi_k}}}{\xi=0}x_jy_k,
\end{equation}
and
\begin{equation}\label{Ci}
C_i(x,y) = \sum_{j,k,l=1}^n \evalat{\diffp{F_i(\xi)}{{\xi_j}{\xi_k}{\xi_l}}}{\xi=0}x_jy_kz_l,
\end{equation}
where $i = 1,2,\dots,n$.

\subsubsection*{Flip bifurcations}

In the case of a flip bifurcation, $A$ has a simple critical eigenvalue $\mu_1 = -1$, and the corresponding critical eigenspace $T^c$ is one-dimensional and spanned by an eigenvector $q\in\mathbb{R} ^n$ such that $Aq=\mu_1q$. Let $p$ be the \textit{adjoint} eigenvector, that is $A^Tp=\mu_1p$. Normalize $p$ with respect to $q$ so that $\langle p,q\rangle = 1$, where $\langle .,.\rangle$ is the standard scalar product in $\mathbb{R} ^n$. 

The critical normal form coefficient $c$, that determines the nondegeneracy of the flip bifurcation and allows us to predict the direction of bifurcation of the period-two cycle, is given by the \textit{invariant formula}
\begin{equation}\label{invariant formula}
c = \frac{1}{6}\langle p , C(q,q,q) \rangle - \frac{1}{2}\langle p , B(q,(A-I_n)^{-1}B(q,q)) \rangle.
\end{equation}

\subsubsection*{The Neimark-Sacker bifurcation} 

The third nondegeneracy condition \ref{generic NS condition 3} can be computed as
\begin{multline}\label{NS d}
d=\frac{1}{2}\text{Re}\bigg(e^{-i\theta_0}\bigg[\langle p , C(q,q,\bar{q}) \rangle + 2\langle p , B(q,(A-I_n)^{-1}B(q,\bar{q})\rangle\\
+ \langle p , B(\bar{q} , (e^{2i\theta_0}I_n-A)^{-1}B(q,q))\rangle\bigg]\bigg),
\end{multline}
where $q$ now is a \textit{complex} eigenvector corresponding to $\mu_1=e^{i\theta_0}$:
$$
Aq=e^{i\theta_0}q,~A\bar{q}=e^{-i\theta_0}\bar{q},
$$
where $\bar{q}$ is the vector of complex conjugates of the elements in $q$.

Note that the numbers \(c\) and \(d\) are also called the first Lyapunov coefficients.  Their size can be different using different methods but their sign is invariant.

\subsubsection*{List of codimension 2 bifurcations in $\mathbb{R} ^2$}

In our coming analysis we will consider a two-dimensional dynamical system, so we consider a two-dimensional, two-parameter discrete-time dynamical system
\begin{equation}\label{codim 2}
x\mapsto f(x,\alpha)
\end{equation}
with $x\in\mathbb{R} ^2$ and $\alpha = (\alpha_1,\alpha_2)^T$ and $f$ sufficiently smooth in $(x,\alpha)$ e.g. $f\in C^1$. Suppose that at $\alpha = \alpha_0$, the system \eqref{codim 2} has a fixed point $x^*$ for which the condition for fold, flip or Neimark-Sacker bifurcation is satisfied. Then there are eight degenerate cases that may occur.

\begin{enumerate}%[(1)]
\item
$\mu_1=1,b=0$ \textit{(cusp)}
\item
$\mu_1=-1,c=0$ \textit{(generalized flip)}
\item
$\mu_{1,2} = e^{\pm i\theta_0},d=0$ \textit{(Cheniciner bifurcation)}
\item
$\mu_1=\mu_2 = 1$ \textit{(1:1 resonance)}
\item
$\mu_1=\mu_2 = -1$ \textit{(1:2 resonance)}
\item
$\mu_{1,2} = e^{\pm i\theta_0} , \theta_0 = \frac{2\pi}{3}$ \textit{(1:3 resonance)}
\item
$\mu_{1,2} = e^{\pm i\theta_0} , \theta_0 = \frac{\pi}{2}$ \textit{(1:4 resonance)}
\item 
$\mu_1=1,\mu_2=-1$ \textit{(fold-flip bifurcation)}
\end{enumerate}

\subsection{Forward positivity of the system} 

Next we turn to the important matter of positive invariance. When using compartmental models in epidemiology, it is nonsensical to have trajectories with negative values. Due to the logistic growth we cannot hope that any initial point in the positive quadrant will remain there. We can however give sufficient conditions on the parameters so that there exists a compact subset of the positive quadrant that preserve non-negativity. Let us denote the mappings in \eqref{syst} by $f$ and $g$ respectively, so that $S_{n+1}=f(S_n,I_n)$ and $I_{n+1}=g(S_n,I_n)$. 

\begin{lemma}\label{f<1}
If $r\leq4$ then $f(x,y)<1$ for all $x,y\geq0$
\end{lemma}
\begin{proof}  
The function \(x(1-x)\leq 1/4\) and \(f(x,y)\leq r x(1-x)\) for non-negative \(x\) and \(y\).  Therefore if \(r\leq4\) \(f(x,y)<1\). 
\end{proof}
The next lemma gives an upper bound for the sum $S+I$, as well as $S$ and $I$ for all $n$.
\begin{lemma}\label{u^*}
The sum $S_n+I_n$ is bounded above by $\frac{(r-1+K)^2}{4Kr}$ for suitable choice of initial conditions.
\end{lemma}
\begin{proof}
We have
\begin{align*}
	S_{n+1}+I_{n+1} 
	&= rS_n(1-S_n)+(1-K)I_n \\
	&= rS_n(1-S_n) - (1-K)S_n+(1-K)(S_n+I_n)\\
	&= S_n( r-1+K-rS_n)+(1-K)(S_n+I_n)\\
	&\leq \frac{ r-1+K}{2r}( r-1+K-r\frac{ r-1+K}{2r}) + (1-K)(S_n+I_n)\\
	&= \frac{( r-1+K)^2}{4r} + (1-K)(S_n+I_n).
\end{align*}
Define $U_n=S_n+I_n$. By the above we have $U_{n+1}\leq \frac{( r-1+K)^2}{4r}+(1-K)U_n$. Now consider the dynamical system 
\begin{equation*}\label{u}
	u_{n+1} = \frac{( r-1+K)^2}{4r} + (1-K)u_n.
\end{equation*}
It has the globally asymptotically stable fixed point $u^* = \frac{( r-1+K)^2}{4Kr}$ since $(1-K)<1$. Hence $U_n\leq u_n\leq u^*$ if $U_0\leq u_0\leq u^*$.
\end{proof}
Denote \({\mathbb R}_\ge^2\) the set of all nonnegative points in \({\mathbb R}^2\).  To determine some   sufficient conditions (in terms of the parameters)for positive trajectories we consider two "generic" sets studied in this paper: 
\begin{enumerate}
\item \(\Omega_1\),  the trianlel with vertices  \((0,0)\), \((0,u^*)\) and \((u^*,0)\), if \(u^*\le1\);
\item  \(\Omega_2\),  the compact set bounded by the curves
\[
\begin{aligned}
C_1^2&=\{(x,y)\in{\mathbb R}_{\ge}^2: x=0, 0\le y\le u^*\},\\
C_2^2&=\{(x,y)\in{\mathbb R}_{\ge}^2: y=u^*-x,  0\le x\le \bar x\},\\
C_3^2&=\{(x,y)\in{\mathbb R}_{\ge}^2: y=\frac{r}{\beta}(1-x)(1+ax), \bar x \le x\le 1\},\\
C_4^2&=\{(x,y)\in{\mathbb R}_{\ge}^2: y=0, 0\le x\le 1\}\\
\end{aligned}
\]
where \(\bar x \) is the intersection point of the curve \(C^2_2\) and \(C^2_3\) which is between \(\frac12\) and \(1\),  if \(1<u^*<\frac{r}{\beta}\);

\item  \(\Omega_3\),  the compact set bounded by the curves
\[
\begin{aligned}
C_1^3&=\{(x,y)\in{\mathbb R}_{\ge}^2: x=0, 0\le y\le \frac{r}{\beta}\},\\
C_2^3&=\{(x,y)\in{\mathbb R}_{\ge}^2: y=u^*-x,  0\le x_1\le x_2\},\\
C_3^3&=\{(x,y)\in{\mathbb R}_{\ge}^2: y=\frac{r}{\beta}(1-x)(1+ax), 0\le x\le x_1\, x_2\le x  \le 1\},\\
C_4^3&=\{(x,y)\in{\mathbb R}_{\ge}^2: y=0, 0\le x\le 1\}\\
\end{aligned}
\]
where \(x_1 \) and \(x_2\) are  the intersection points of the curve \(C_2^3\) and \(C_3^3\) which lie in \((0,\frac12)\) and \((\frac12,1)\) respectively, if \(1<\frac{r}{\beta}<u^*\).
\end{enumerate}

Note that \(u^*>1\) is the same as \(r<(1+\sqrt{K})^2\), thus we will use them interchangeably.  Note also that our conditions does not cover all cases.

The following proposition gives sufficient conditions for positive trajectories for any initial state in the specified region. Its proof is given in Appendix \ref{proof}.
\begin{proposition}\label{Omega positive}
%If $\sqrt{K}+1\leq r \leq(\sqrt{K}+1)^2, (S_0,I_0)\in\Omega$ and either $\beta<r$ or $r<\beta<\frac{4Kr^2}{(r-1+K)^2}$, then $(S_n,I_n)\in\Omega$ for all $n$.
Assume \(0<K<1\).  
\begin{enumerate}
\item Assume \(\sqrt{K}+1\le r \le(\sqrt{K}+1)^2\) and either \(\beta<r\) or \(r<\beta<\frac{4Kr^2}{(r-1+K)^2}\).  
If \((S_0,I_0)\in\Omega_1\) then  \((S_n,I_n)\in\Omega_1\) for all \(n\).

\item Assume \((\sqrt{K}+1)^2<r\le 4\)  and \(\beta<\frac{r}{2u^*-1}\).  
If \((S_0,I_0)\in\Omega_2\) then  \((S_n,I_n)\in\Omega_2\) for all \(n\).

\item  Let \(a=1\).  Assume either that \(u^*\ge2\) and \(\beta<r\) or that \(\frac54<u^*<2\) and \(\beta<r v_+\) where \(v_+=\left( \frac{\sqrt{u^*+1} +\sqrt{u^*-1}}{2}\right)^2>0\). Then   \((S_n,I_n)\in\Omega_3\) for all \(n\) if  \((S_0,I_0)\in\Omega_3\).

\end{enumerate}
\end{proposition}

Note that the second item corresponds to the case \(1<u^*<\frac{r}{\beta}\) which is implied by the condition on \(\beta\) and the third deals with a special case when \(u^*>\frac{r}{\beta}>1\).  
We point out that \(f(x,y)\) can still be positive when \(r\) is larger than \(4\).  This can be seen by the following reformulation, for example in the last item of the above Proposition, of the conditions as follows. 
A straightforward and tedious calculation shows that 
\(\frac54<u^*\) is equivalent to \(r>\frac{2+3K+\sqrt{5K^2+20K}}{2}<(1+\sqrt{K})^2\).  So for the bounds of \(r\) in terms \(K\).  
\[
\frac54<u^*<2 \iff \frac{2+3K+\sqrt{5K^2+20K}}{2}<r<3K+1+2\sqrt{2K(K+1)}.
\]

\section{Stability analysis of fixed points and description of bifurcation points}

Solving the equations 
\(S=f(S,I),I=g(S,I)\),
yields two points \((S^*,I^*)\) of the system defined by \eqref{syst}: 
\(E_0 = (\frac{r-1}{r},0)\), and 
 \(E_1 = (\frac{K}{\beta-aK},\frac{r-1}{\beta-aK}-\frac{rK}{(\beta-aK)^2})\).  The fixed point \(E_0\) is called {\em disease free} and \(E_1\) 
 is called {\em endemic} in epidemiology. 

In the epidemical modelling we have the positive restrictions on \(S^*,I^*\).
This means  \(E_0\ge0\) that is \(r>1\);  and \(E_1\ge0\) if and only if \(\beta-aK>0\) and \(\frac{r-1}{\beta-aK}-\frac{rK}{(\beta-aK)^2}\ge 0\).  The last 
inequality implies \(\beta-aK>0\)  if \(r>1\).  So \(E_1\) is nonnegative if 
\[
r>1 \text{ and } \ \beta> \beta_0:=\frac{K(r+a(r-1))}{r-1} 
\]
Hence we have proved the following proposition.

\begin{proposition} The SIR model defined by \eqref{syst} has at most two non-negative fixed points: The disease free fixed point \(E_0=\left(\frac{r-1}{r},0\right)\) and the endemic fixed point \(E_1=\left(\frac{K}{\beta-aK}, \frac{r-1}{\beta-aK}-\frac{rK}{(\beta-aK)^2}\right)\).  
More precisely,
\begin{itemize}
\item there is one fixed point \(E_0\) if \(r>1\), and \(\beta<\beta_0\); 
\item there is one fixed point \(E_0\) and 
\(E_1\), if \(r>1\) and \(\beta>\beta_0\) . 
 \end{itemize}
\end{proposition}

\subsection{Asymptotical analysis of the fixed points}

Now we turn to local stability analysis of the dynamical system \eqref{syst}.  

By standard procedure in stability analysis, we compute the Jacobian matrix evaluated at each fixed point and determine the location of its eigenvalues.  
For our dynamical system (\ref{syst}), the Jacobian matrix is
$$
J(S,I) = 
\begin{pmatrix}
r-2rS - \bigg(\frac{\beta I}{1+aS}-\frac{a\beta SI}{(1+aS)^2}\bigg) & -\frac{\beta S}{1+aS} \\
\frac{\beta I}{1+aS}-\frac{a\beta SI}{(1+aS)^2} & 1-K+\frac{\beta S}{1+aS}
\end{pmatrix}
$$
which is simplified to
\begin{equation*}
J(S,I) = 
\begin{pmatrix}
r-2rS - \frac{\beta I}{(1+aS)^2} & -\frac{\beta S}{1+aS} \\
\frac{\beta I}{(1+aS)^2} & 1-K+\frac{\beta S}{1+aS}
\end{pmatrix}.
\label{Jacobian}
\end{equation*}

\subsection*{Stability of disease free fixed point}

At the disease-free fixed point, the Jacobian matrix is
$$
J(E_0) = J\left (\left(\frac{r-1}{r},0\right)\right)=
\begin{pmatrix}
2-r & -\frac{\beta(r-1)}{r+a(r-1)}\\
0 & 1-K+\frac{\beta(r-1)}{r+a(r-1)}
\end{pmatrix}.
$$
Clearly the eigenvalues are
\begin{align*}
\lambda_1 &= 2-r\\
\lambda_2 &= 1-K+\frac{\beta(r-1)}{r+a(r-1)}.
\end{align*}

To find out when $E_0$ is stable, we must solve the system of inequalities
$$
\begin{cases}
|2-r|<1\\
|1-K+\frac{\beta(r-1)}{r+a(r-1)}|<1
\end{cases}
\iff
\begin{cases}
1<r<3\\
-2 < \frac{\beta(r-1)}{r+a(r-1)}-K<0.
\end{cases}
$$

Then we have
$$ 
\begin{cases}
1<r<3\\
K-2 < \frac{\beta(r-1)}{r+a(r-1)}<K
\end{cases}
\iff
\begin{cases}
1<r<3\\
\frac{(K-2)(r+a(r-1))}{r-1}<\beta<\frac{K(r+a(r-1))}{r-1}.
\end{cases}
$$
Now, since $\beta$ is the coefficient for the force of infection, it must be positive. It is clear, since $K%=\mu+\gamma 
<1$ that the lower bound for $\beta$ is negative. So, to summarize, if $1<r<3$ and $0<\beta< \beta_0$, where
$$
\beta_0=\frac{K(r+a(r-1))}{r-1},
$$ 
then $E_0$ is locally asymptotically stable.

Note that if one of the conditions is violated but not on the boundaries then \(E_0\) is a saddle point.  That is, \(E_0\) is a saddle point if  \(r>3\) or \(r<1\) together with \(\beta<\beta_0\), or \(1<r<3\) and \(\beta>\beta_0\). 

\subsection*{Stability of endemic fixed point}

This is a more complicated case.   Recall that it is required that \(r>1\) and 
\(\beta>\beta_0\) for  $E_1$ being positive.

The Jacobian matrix evaluated at $E_1$ is
$$
J(E_1) = 
\begin{pmatrix}
\frac{2 K r}{a K-\beta }+\frac{K (a (r-1)+r)}{\beta }+1 & -K \\
r+\frac{K (a-(a+1) r)}{\beta }-1 & 1 \\
\end{pmatrix},
$$
whose characteristic polynomial is 
$$
p(z) = z^2 - p_1z+p_0
$$
where $p_1=-\text{trace}(J(E_1))$ and $p_0=\det(J(E_1))$. More precisely
\begin{align*}
p_1 &= -\frac{2 K r}{a K-\beta }-\frac{K (a (r-1)+r)}{\beta }-2\\
p_0 &= 1+K \left(\frac{2 r}{a K-\beta }-\frac{(K-1) (a (r-1)+r)}{\beta }+r-1\right).
\end{align*}
By \eqref{stability conditions trace det}
$$
\begin{aligned}
|p_1|&<1+p_0\\
|p_0|&<1 \\
r&>1\\
\beta&>\frac{K(r+a(r-1))}{r-1}.
\end{aligned}
$$

With help of Mathematica, we get the following parameter constraints
$$
\begin{cases}
1<r\leq3\\
\beta_0<\beta<\beta_2
\end{cases}
~\text{or}~
\begin{cases}
3<r<r_{max}\\
\beta_1<\beta<\beta_2
\end{cases}
$$
where
\begin{align*}
\beta_0 =&\tfrac{K(r+a(r-1))}{r-1}\\
\beta_1 =& \frac{1}{2}\bigg(\tfrac{K(2a(3+K(r-1)-r)+(K+2)r)}{4+K(r-1)}+ \sqrt{\tfrac{K^2((K+2)^2r^2+4a^2(r+1)^2+4ar(14-5K-2r+3Kr))}{(4+K(r-1))^2}} \bigg)\\
\beta_2 =&\frac{1}{2}\bigg(a(2K-1) + \tfrac{r(K+1)}{r-1} + \sqrt{a^2+\tfrac{2ar(3K-1)}{r-1} + \tfrac{r^2(K+1)^2}{(r-1)^2}}\bigg)\\
r_{max} = &\frac{1}{2} \sqrt{\tfrac{16 a^2+88 a K-32 a+25 K^2+40 K+16}{K^2}}+\tfrac{4 a+5 K+4}{2 K}.
\end{align*}

Furthermore, \(E_1\) is a saddle point if (i)  \(1<r\le3\) and \(\beta<\beta_0\), or (ii) \(1<r\le3\), or
\(\beta>\beta_2\) or (iii) \(3<r<r_{max}\) and \(\beta<\beta_1\), or (vi)  \(3<r<r_{max}\) and 
\(\beta>\beta_2\).
\\

The above discussion proves the following theorem.

\begin{theorem} The SIR model defined by \eqref{syst} has the following stability properties:
\begin{enumerate}
\item For  
$$
\begin{cases}
1<r<3\\
0<\beta<\beta_0
\end{cases}
$$
then the disease-free equilibrium $E_0$ is locally asymptotically stable. Finally, 
\item if
$$
\begin{cases}
1<r\leq3\\
\beta_0<\beta<\beta_2
\end{cases}
~\text{ or }~
\begin{cases}
3<r<r_{max}\\
\beta_1<\beta<\beta_2
\end{cases}
$$
then the endemic equilibrium $E_1$ is locally asymptotically stable.
\end{enumerate}
\end{theorem}
It is depicted in Figure~\ref{stability}.
\begin{figure}[!h]
\centering
\includegraphics[width=.9\linewidth]{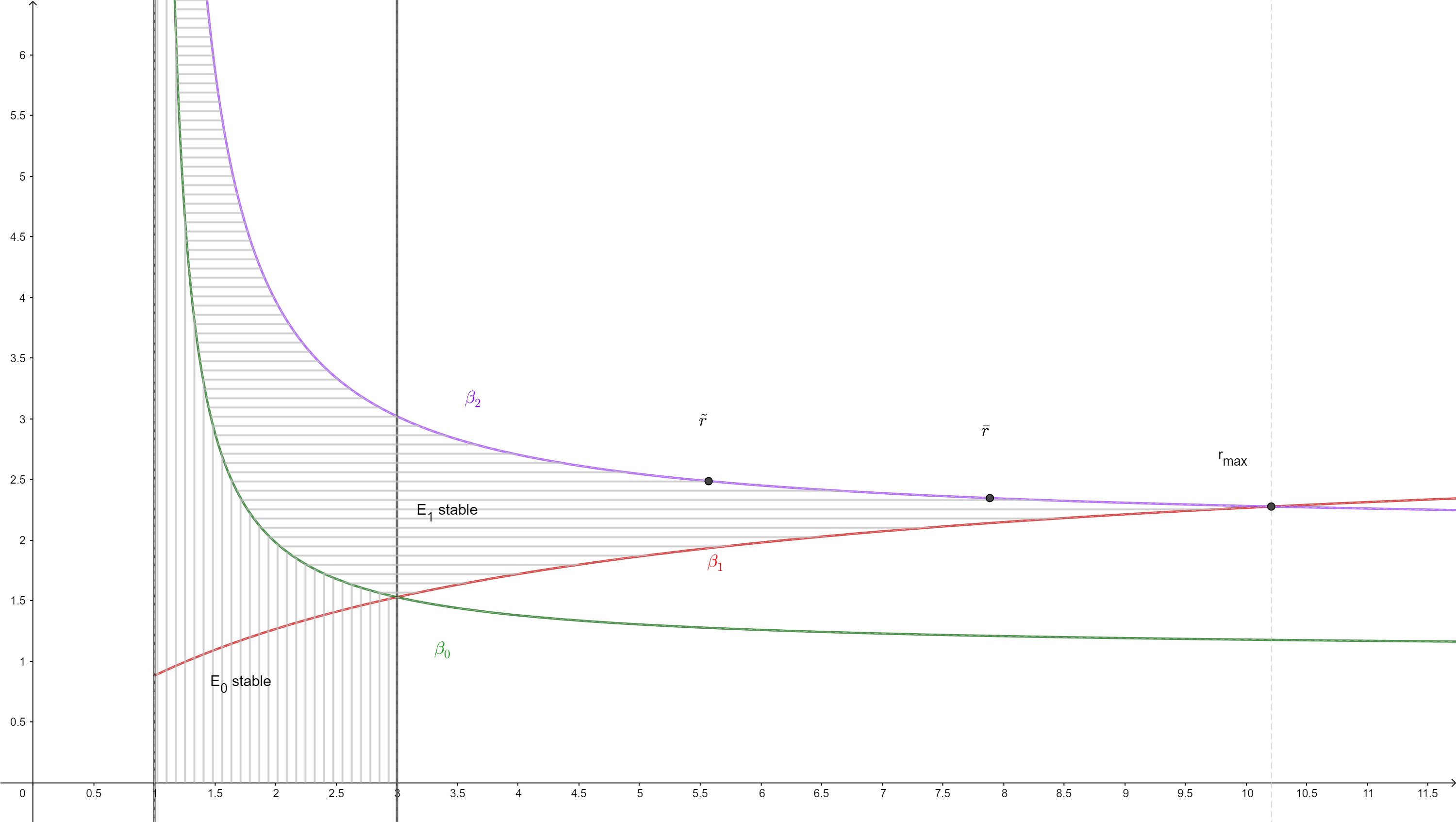}
\caption{Stability regions on  $r\beta$-plane with $a=0.2$, $K=0.9$.}
\label{stability}
\end{figure}

%\subsection{Local positive invariance}
%
%Note that the discrete system under consideration is not a positive system because the Jacobian 
%matrix is not positive.   This implies that  the positive quadrant is not positively invariant with respect to 
%system motions due to the nature of the discrete system.  However we can show  that there is a 
%domain of asymptotical stability of the state \(E_1\) which lies completely in the positive quadrant.  
%Then this domain is positively invariant with respect to the system motions, \cite{G-function}.  As for 
%the disease-free \(E_0\), it is on the boundary of the positive quadrant.  Clearly the neighbourhood of 
%\(E_0\), thus the domain of asymptotical stability is not entirely in the positive quadrant.  Therefore it is 
%much likely the state can become negative even when the initial value is very closed to \(E_0\).   From 
%epidemiological point of view, the trajectory becoming negative means that there are not enough 
%amount of susceptible population  that can be infected further.  This has the same effect as the 
%disease dies  out when infected tends to \(0\).  In fact such situations may occur in the classic SIR model for 
%flu.  See e.g. \cite{FirstcourseMM}.  To summarise we can say that if there is no finite time escape (to the negative values) then the trajectory will aprpoach to \(E_0\).

\subsection{Bifurcation points}

We have found conditions on the parameters $r$ and $\beta$ for stability of \eqref{syst}. 
Now we investigate how the dynamics of the system \eqref{syst} changes under variation of these parameters.  
In this section, we discuss flip bifurcation, which is one of the codimension~1 bifurcation, the bifurcations that depend on just one parameter based on Definition~\ref{def:codim1}.
Since bifurcation analysis relies  heavily on  on the eigenvalues of the Jacobian of each fixed point  at the bifurcation points,
we present our findings of eigenvalues on the boundaries of stability regions as follows.   \\

\noindent \textbf{Disease free fixed point}, $E_0 = (\frac{r-1}{r},0)$.  In this case
the stability conditions were $1<r<3$ and $0<\beta<\beta_0$. Recall however that the lower bound for $\beta$ was derived under the biological constraint that $\beta$ has to be non-negative. Hence 0 is not mathematically the lower bound for stability and can therefore be ignored here. The conditions, with this in mind, can be violated as follows:

\begin{enumerate}\renewcommand{\labelenumi}{(\roman{enumi})}
\item
$r = 1, 0<\beta<\beta_0 \implies \lambda_1 = 1,|\lambda_2|<1$
\item
$r = 1,\beta = \beta_0 \implies \lambda_1 = 1, \lambda_2 = -1$
\item
$r = 3, 0<\beta<\beta_0 \implies \lambda_1 = 1, |\lambda_2| <1$
\item
$r = 3, \beta = \beta_0 \implies \lambda_1 = 1, \lambda_2 = -1$
\item
$1<r<3, \beta = \beta_0 \implies \lambda_1 = 1, |\lambda_2| <1$
\end{enumerate}

\noindent \textbf{Endemic fixed point}, $E_1 = (\frac{K}{\beta-aK},\frac{r-1}{\beta-aK} -\frac{rK}{(\beta-aK)^2})$: Now
the stability conditions were $1<r\leq3$ and $\beta_0<\beta<\beta_2$ or $3<r<r_{max}$ and $\beta_1 <\beta<\beta_2$. Note that when $r = 3$, we get $\beta_0=\beta_1$, and when $r=r_{max}$ we have $\beta_0 = \beta_2$. In fact, we can also have $\beta_0 = \beta_2$ but only when $r = 0$ or $r = \frac{a}{a+1}<1$ so it has no effect here. The stability conditions can be violated as follows:

\begin{enumerate}\renewcommand{\labelenumi}{(\roman{enumi})}
\item
$1<r<3,\beta = \beta_0 \implies \lambda_1=1,|\lambda_2|<1$
\item
$1<r<3,\beta = \beta_2 \implies \lambda_{1,2} = e^{\pm i \theta_0}, 0<\theta_0<\pi$
\item
$r = 3, \beta = \beta_0 \implies \lambda_1 = 1, \lambda_2 = -1$
\item
$r = 3, \beta = \beta_2 \implies \lambda_{1,2} = e^{\pm i \theta_0}, 0<\theta_0<\pi$
\item
$r = r_{max},\beta= \beta_2 \implies \lambda_1 = -1 ,\lambda_2 = -1$ 
\item
$3<r<r_{max}, \beta = \beta_1 \implies \lambda_1 = -1, |\lambda_2|<1$
\item
$3<r<r_{max}, \beta = \beta_2 \implies \lambda_{1,2} = e^{\pm i \theta_0}, 0<\theta_0<\pi$
\end{enumerate}

Co-dimension 2 bifurcations occur when the non-degeneracy  conditions are violated.    By identifying the list of the eight co-dimension 2 bifurcations and the eigenvalues listed above  we can conclude:
\begin{itemize}
\item {\bf Bifurcations from $E_0$}

There is  a fold-flip bifurcation when $r=3,\beta=\beta_0$.

\item {\bf Bifurcations from $E_1$}

We have 1:2, 1:3 and 1:4 resonances when $\beta = \beta_2$ and $r = r_{max},\tilde{r},\bar{r}$ respectively. 
They are depicted in Figure~\ref{stability}.  Apart from that, there is a fold-flip bifurcation at $r=3,\beta = \beta_0$ (=\(\beta_1\)).\\

\end{itemize}

In next two sections we investigate the degeneracy of the flip and the Neimark-Sacker bifurcations.

\section{Analysis of flip bifurcation}

\noindent \textit{Bifurcations from $E_0$:}  At $\beta = \beta_0$ for all $1<r<3$, there is a fold bifurcation, and $E_0$ loses 
stability to \(E_1\) if  \(\beta\) increases and  passes  \(\beta_0\) for all $1<r<3$.  Moreover 
there is a flip bifurcation at $r = 3$ for all $\beta<\beta_0$. In this case \(E_0\) loses stability to some periodic orbits which we will show later by showing this 
flip bifurcation is generic and stable. 
Note that these statements coincide with the remark on \(E_0\) being a saddle point made in the previous section.  \\

\noindent \textit{Bifurcations from $E_1$:} For $1<r<3$ and $\beta = \beta_0$, there is a fold bifurcation, and $E_1$ loses stability to $E_0$. When $\beta=\beta_1$ and $3<r<r_{max}$, there is a flip. For $1<r<r_{max}$ and $\beta = \beta_2$ there is a Neimark-Sacker bifurcation, except for some degenerate cases which we deal with later.\\

These can be seen in Figure~\ref{stability}.
Now we turn to study the genericity conditions on some of these bifurcation points. This is somewhat technical, and include some rather lengthy computations which are presented in the appendices.
\subsection{Periodic-doubling bifurcation from $E_0$}

First we prove the following proposition.
\begin{proposition}
 Assume $0<\beta < \frac{1}{2}(3K+2aK)$.  Then there is a flip bufurcation from \(E_0\) at \(r=3\).  
\end {proposition}
\begin{proof}At $E_0$, for $r=3$ the Jacobian matrix is
$$
A = J(E_0) =
\begin{pmatrix}
2-r & -\frac{\beta(r-1)}{r+a(r-1)}\\
0 & 1-K+\frac{\beta(r-1)}{r+a(r-1)}
\end{pmatrix}
=
\begin{pmatrix}
-1 & -\frac{2\beta}{3+2a}\\
0 & 1-K+\frac{2\beta}{3+2a}.
\end{pmatrix}
$$
The eigenvalues of $A$ are $\mu_1 = -1$ and $\mu_2 = 1-K+\frac{2\beta}{3+2a}$. Now, $|\mu_2|<1$ if and only 
if $0<\beta < \frac{1}{2}(3K+2aK) = \beta_0\left|_{r=3}\right.$.
\end{proof}

This means the dynamical system undergoes a flip bifurcation, which is a periodic doubling bifurcation, resulting a 2-periodic orbit.  Next we investigate the stability of this \(2\)-periodic orbit.  The answer can be found if we can check the conditions in Theorem~\ref{generic flip}. 

\begin{theorem} The flip bifurcation found in preceding proposition is generic and the resulting \(2-\)periodic orbit is stable for \(3<r<1+\sqrt{6}\).
\end{theorem}

\begin{proof} Following the procedure  outlined after Theorem~\ref{generic flip}, we compute an eigenvector $q$ of $A$ associated with $\mu_1=-1$. We have
$$
Aq = -q \iff (A+I_2)q = 0 \iff 
\begin{pmatrix}
0 & -\frac{2\beta}{3+2a}\\
0 & 2-K+\frac{2\beta}{3+2a}
\end{pmatrix}
\begin{pmatrix}
q_1\\
q_2
\end{pmatrix}
= 0.
$$

We may choose $q_1 = 1,q_2=0$ to get the eigenvector $q=(1~0)^T$. Next, we compute an adjoint eigenvector $p$, normalized with respect to $q$, so that $\langle p,q \rangle = 1$. Fortunately, we see that $p$ must take the form $p=(1~p_2)^T$. Then we can find $p_2$ by
$$
A^Tp = -p \iff (A^T+I_2)p = 0 \iff 
\begin{pmatrix}
0 & 0\\
-\frac{2\beta}{3+2a} & 2-K+\frac{2\beta}{3+2a}
\end{pmatrix}
\begin{pmatrix}
1\\
p_2
\end{pmatrix}
= 0.
$$
This implies that 
$$
p_2 = \frac{\frac{2\beta}{3+2a}}{2-K+\frac{2\beta}{3+2a}} =\frac{2 \beta }{2 \beta+(2 a+3) (2-K)}.
$$

Our goal is to compute 
$$
c = \frac{1}{6}\langle p , C(q,q,q) \rangle - \frac{1}{4}\langle p , B(q,(A-I_2)^{-1}B(q,q)) \rangle,
$$
which first requires the computation of $B(x,y)$ and $C(x,y,z)$. As this computation is quite tedious and of no immediate interest, we just move on to state that $c = 9 > 0$ which implies that the flip is generic and the resulting 2-cycle is stable. The interested reader is referred to appendix \ref{compute c} for the details of the computation.

To determine the upper bound for \(r\) we study the map formed by the second iterate, i.e.
\(
\begin{pmatrix}  f(f(S,I),g(S,I))\\ g(f(S,I),g(S,I))\end{pmatrix}
\)
It has two nontrivial fixed points in addition to the fixed point found earlier:
$$
(S,I)=(\tfrac{1+r\pm\sqrt{(r-3)(r+1)}}{2r},0),
$$
By similar argument and computation as for stability analysis for fixed points we can show that they are stable for $3<r<1+\sqrt{6}$.  Hence the 2-periodic orbit is stable for \(3<r<1+\sqrt{6}\).
\end{proof}

Note that we also found that both fixed points yield the eigenvalues of the Jacobeans
\begin{equation}\label{secondflipeigs}
\begin{aligned}
\mu_1 &= 4 - r(r-2)\\
\mu_2 &= \tfrac{(K-1)^2 \left(a^2 (r+1)+a r (r+1)+r^2\right)-\beta (K-1) (r+1) (2 a+r)+\beta ^2 (r+1)}{a^2 (r+1)+a r (r+1)+r^2}.
\end{aligned}
\end{equation}
At $r=1+\sqrt{6}$, we find that $\mu_1=-1$, so there is a flip in both cases. We will now show that the flip is generic and the resulting 4-periodic orbit is stable. We consider only the case with negative square root since the computations for the other one are almost exactly the same.

Again, we look for an eigenvector of the Jacobian matrix $J_2$ of the second iterate at $r=1+\sqrt{6}$ which is quite easy since then, 
$$
A = J_2(\tfrac{1+r\pm\sqrt{(r-3)(r+1)}}{2r},0) = 
\begin{pmatrix}
-1 & a_{12}\\
0 & a_{21}
\end{pmatrix},
$$
where, if we denote by $a^{\pm}_{12}$ the off-diagonal element in the case of positive and negative square roots respectively we have
$$
\begin{aligned}
a^-_{12} &= \tfrac{\beta \left(\left(2 \left(\sqrt{6}+2\right) a+\sqrt{2}+2 \sqrt{3}+3 \sqrt{6}+8\right) K+4 \left(\sqrt{2}+\sqrt{3}\right) a-2 \left(\left(\sqrt{6}+2\right) \beta -3 \sqrt{2}-\sqrt{3}+\sqrt{6}+1\right)\right)}{2 \left(a \left(\left(\sqrt{6}+2\right) a+3 \sqrt{6}+8\right)+2 \sqrt{6}+7\right)}\\
a^+_{12} &= -\tfrac{\beta \left(-2 \left(\sqrt{6}+2\right) a K+4 \left(\sqrt{2}+\sqrt{3}\right) a+2 \left(\left(\sqrt{6}+2\right) \beta +3 \sqrt{2}+\sqrt{3}+\sqrt{6}+1\right)+\left(\sqrt{2}+2 \sqrt{3}-3 \sqrt{6}-8\right) K\right)}{2 \left(a \left(\left(\sqrt{6}+2\right) a+3 \sqrt{6}+8\right)+2 \sqrt{6}+7\right)}\\
a_{21} &= \tfrac{\beta \left(\left(\sqrt{6}+2\right) \beta -\left(2 \left(\sqrt{6}+2\right) a+3 \sqrt{6}+8\right) (K-1)\right)}{a \left(\left(\sqrt{6}+2\right) a+3 \sqrt{6}+8\right)+2 \sqrt{6}+7}+(K-1)^2.
\end{aligned}
$$

We want to determine $q$ so that
$$
(A+I_2)q = 
\begin{pmatrix}
0 & a_{12}\\
0 & 1+a_{21}
\end{pmatrix}
\begin{pmatrix}
q_1\\
q_2
\end{pmatrix}
= 0.
$$
Hence, we may take $q=(1~0)^T$. Since we require $\langle p , q\rangle = 1$, $p$ must take the form $(1~p_2)^T$. Then we can find $p_2$ by considering
$$
(A^T+I_2)p=
\begin{pmatrix}
0 & 0\\
a_{12}&a_{21}
\end{pmatrix}
\begin{pmatrix}
1\\
p_2
\end{pmatrix}
= 0,
$$
which tells us that 
$$
p_2 = -\frac{a_{12}}{1+a_{21}}.
$$

From here following the same procedure as before we can compute $c$. The computations are completely analogous to what has been shown in appendix \ref{compute c} and we find in the case of the negative square root that
$$
c = -10 \left(\sqrt{2}-2\right) \left(2 \sqrt{6}+7\right) \approx 69.7
$$
and for the positive square root
$$
c = 10 \left(\sqrt{2}+2\right) \left(2 \sqrt{6}+7\right) \approx 406.3.
$$
Hence the flip is generic and the resulting orbit is stable in both cases.  

We expect a continuing period-doubling of the system until \(r\) is close 4 since the system in this case behaves like the logistic mapping.   This is confirmed by  numerical simulation, shown in Figure~\ref{bifurcation E0r}.  
\begin{figure}[!h]
\centering
\includegraphics[width=.9\linewidth]{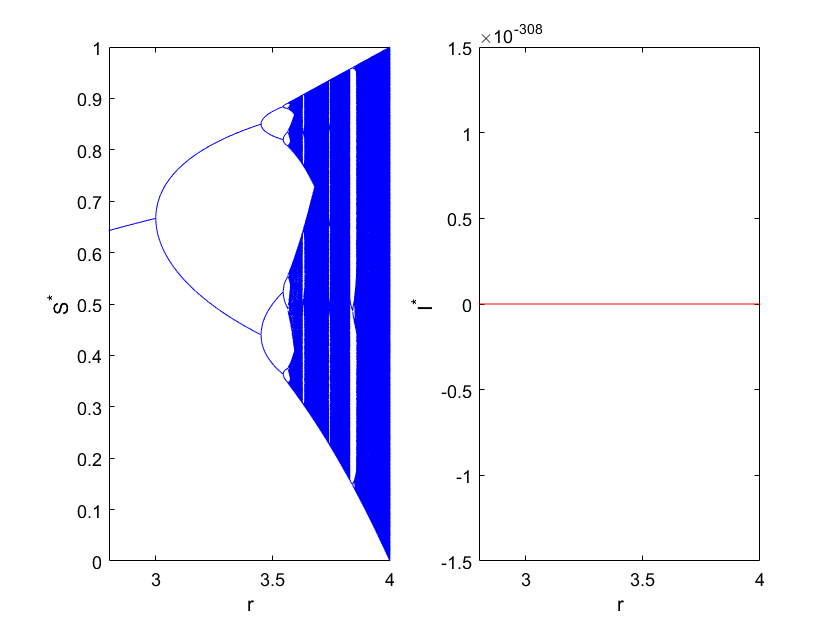}
\caption{Bifurcation diagram with parameter $r$ where $\beta = 1.1 <\beta_0$, \(a=1\) \(K=1/2\).}
\label{bifurcation E0r}
\end{figure}

In fact we can argue that it is true based on Proposition~\ref{p3-p}.  In case \(\beta<(a+1)K(<\beta_0)\) we have 
\(I_n\)  converges to \(0\) as far as \(S_n\) is between \(0\) and \(1\), which is guaranteed by Proposition~\ref{f<1} and \(I_n\) is small.  
Then \(S_n\) behaves like the logistic map if \(I_n\) is close to \(0\).  It remains to argue that this holds too for  \((a+1)K<\beta <\beta_0\).    The main difficulty here is to make sure that \(f(x,y)>0\) since \(f(x,y)<1\) for \(r\le 4\) based on Proposition~\ref{f<1}.  
By Proposition~\ref{p1-p} if we start with \(S_0<\frac{K}{\beta-aK}\) and \(I_0<1\) then \(S_n\) will stay in the interval \((0,1)\)  However if \(\frac{K}{\beta-aK}<S_0<1\) the choice of \(I_0\) is more delicate. Roughly speaking it will work if \(I_0\) is below \(1/\gamma_2<1\).  In a more careful way we can say 
that if \(I_n\to0\)  then \(S_n\) behaves like a logistic mapping.

\subsection{Generic investigation of flip bifurcation from $E_1$}

In a similar manner one can find eigenvectors and compute $c$ for $3<r<r_{max}$ and $\beta = \beta_1$. We denote by $A$ the Jacobian matrix evaluated at $E_1$ when $\beta=\beta_1$. Then 
$$
A = 
\begin{pmatrix}
a_{11} & -K \\
a_{21} & 1
\end{pmatrix}
$$
where
\begin{multline*}
a_{11}=1+\\\tfrac{4 (K (r-1)+4) (4 a+K r)}{-2 (K (r-1)+4) \sqrt{\frac{K^2 \left(4 a^2 (r+1)^2+4 a r (3 K r-5 K-2 r+14)+(K+2)^2 r^2\right)}{(K (r-1)+4)^2}}-(K (r-3)+8) (4 a+K r)+K^2 r^2-5 K^2 r+4 K r},
\end{multline*}
and
\begin{equation*}
a_{21} = \tfrac{2 K (a-(a+1) r)}{\sqrt{\frac{K^2 \left(4 a^2 (r+1)^2+4 a r (3 K r-5 K-2 r+14)+(K+2)^2 r^2\right)}{(K (r-1)+4)^2}}+\tfrac{K (2 a (K (r-1)-r+3)+(K+2) r)}{K (r-1)+4}}+r-1.
\end{equation*}

The first task is to find an eigenvector of $A$ associated with $\mu_1=-1$. Hence, we solve the equation
$$ 
(A+I_2)q = 0
$$
where $q = (q_1 ~ q_2)^T$. This yield
$$
\begin{cases}
(a_{11}+1)q_1-Kq_2=0\\
a_{21}q_1 + 2q_2 = 0
\end{cases}
\iff
\begin{cases}
q_1 = -\frac{2a_{11}}{2+a_{21}K}\\
q_2 = \frac{a_{11}a_{21}}{2+a_{21}K}.
\end{cases}
$$
For convenience we divide both $q_1$ and $q_2$ by $q_1$ to get the eigenvector
$$
q = 
\begin{pmatrix}
1\\
-\frac{a_{21}}{2}
\end{pmatrix}.
$$

Next, we determine the adjoint eigenvector $p = (p_1~p_2)^T$:
$$
(A^T+1)\begin{pmatrix}p_1\\p_2\end{pmatrix} = 
\begin{pmatrix}
(a_{11}+1)p_1 + a_{21}p_2\\
-Kp_1+2p_2
\end{pmatrix}
=
\begin{pmatrix}
0\\
0
\end{pmatrix}.
$$
Together with the constraint that $\langle p , q \rangle = 1$ this yields that
$$
\begin{cases}
(a_{11}+1)p_1 + a_{21}p_2 = 0\\
Kp_1 = 2p_2\\
p_1 - \frac{a_{21}}{2}p_2=1.
\end{cases}
$$

From the second and third equation we get that $p_1 = \frac{4}{4-Ka_{21}}$ and $p_2 = \frac{2K}{4-Ka_{21}}$, and one can check that this fulfils the first equation as well. This gives us the adjoint eigenvector
$$
p = \frac{4}{4-Ka_{21}}
\begin{pmatrix}
1\\
\frac{K}{2}
\end{pmatrix}.
$$
Again, we wish to compute
$$
c = \frac{1}{6}\langle p , C(q,q,q) \rangle - \frac{1}{4}\langle p , B(q,(A-I_2)^{-1}B(q,q)) \rangle,
$$
which first requires the computation of $B(x,y)$ and $C(x,y,z)$. We refer the interested reader to appendix \ref{compute c}. Unfortunately, numerical simulations show that $c$ can take on both positive and negative values depending on $r$.  This is depicted in Figure~\ref{c} where we draw the plane 
at \(c=0\) for a better view of signs.

\begin{figure}[!h]
\centering
\includegraphics[width=.4\linewidth]{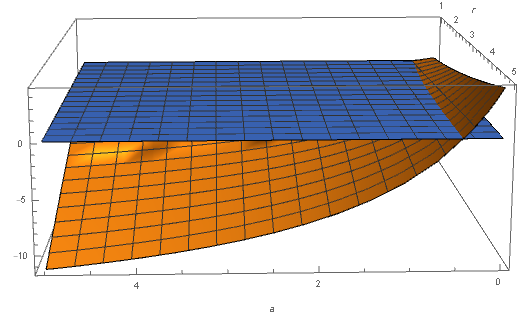}
\caption{Signs of the first Lyapunov coefficient \(c\) for bifurcation from \(E_1\)}
\label{c}
\end{figure}

Figure~\ref{FlipE1-cycles} shows that the flip bifurcation from \(E_1\) results several stable periodic orbits of periods \(2\), \(4\), \(8\) and \(3\) where 
we fix \(a=1\) and \(K=0.5\) and only \(S\)-orbits are plotted  and the simulations show that \(I_n\to 0\). 
\begin{figure}[!h]
\centering
\subfloat{\includegraphics[width=.33\linewidth]{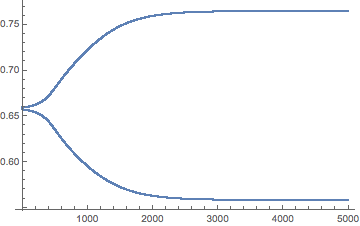}}
\subfloat{\includegraphics[width=.33\linewidth]{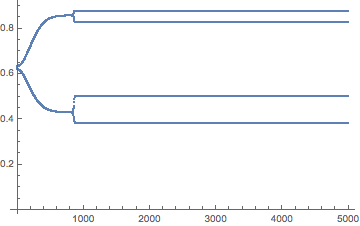}}\\
\subfloat{\includegraphics[width=.33\linewidth]{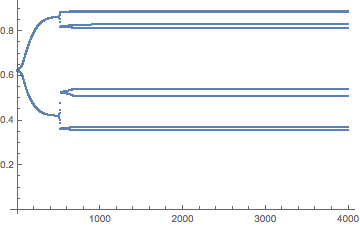}}
\subfloat{\includegraphics[width=.33\linewidth]{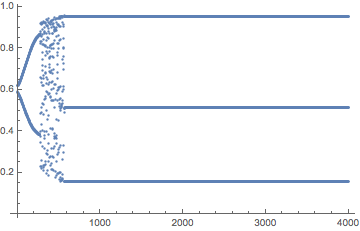}}
\caption{Flip bifurcation of from \(E_1\): (a) \(r=3.1,\beta=1.26\) ;(b) \(r=3.5,\beta=1.3\); (c) \(r=3.55,\beta=1.3\); (d) \(r=3.82843,\beta=1.33\) }
\label{FlipE1-cycles}
\end{figure}

\section{Analysis of Neimark-Sacker bifurcation}

In this section we will investigate the nondegeneracy conditions to see whether the Neimark-Sacker bifurcation is generic.   Note that it occurs at $\beta = \beta_2$.

First, for $\beta=\beta_2$, the Jacobian matrix is
$$
A=
\begin{pmatrix}
a_{11} & -K\\
a_{21} & 1
\end{pmatrix}
$$ 
where
\begin{align*}
a_{11} &= \tfrac{K \left(-r \sqrt{a^2+\frac{2 a (3 K-1) r}{r-1}+\frac{(K+1)^2 r^2}{(r-1)^2}}+\sqrt{a^2+\frac{2 a (3 K-1) r}{r-1}+\frac{(K+1)^2 r^2}{(r-1)^2}}+K r+r\right)+a (K r+K-2)}{2 a (K-1)}\\
a_{21} &= \tfrac{2 K (r+a(r-1))}{\sqrt{a^2+\frac{2 a (3 K-1) r}{r-1}+\frac{(K+1)^2 r^2}{(r-1)^2}}+a (2 K-1)+\frac{(K+1) r}{r-1}}+r-1.
\end{align*}

The characteristic polynomial is
$$
P_A(z) = z^2 - \text{tr}(A)z + \det(A),
$$
and using standard relations between coefficients and zeros of a degree two polynomial we get that
\begin{equation}\label{coeff/roots}
\begin{cases}
\mu_1 + \mu_2 = a_{11} + 1\\
\mu_1\mu_2 =\det(A) = a_{11}+a_{21}K = 1.
\end{cases}
\end{equation}
We have used that the zeros sum to negative the coefficient of $z$, and that the product is equal to the constant term. It is a simple but tedious matter to check that $\det(A) = 1$. Knowing that one eigenvalue lies on the unit circle, we immediately get that the other one must do so as well, for otherwise their product could not be 1. This also excludes the case $\mu_{1,2}=\pm 1$ so we must have complex conjugate eigenvalues
$$
\mu_{1,2} = e^{\pm i\theta_0} = \sigma \pm i\omega.
$$
From \eqref{coeff/roots} it is clear that $\mu_1+\mu_2=2\sigma = a_{11}+1$, and specifically we get
\begin{equation*}\label{real part}
\sigma 
=\tfrac{K \left(\sqrt{a^2+\tfrac{2 a (3 K-1) r}{r-1}+\tfrac{(K+1)^2 r^2}{(r-1)^2}}+K r+r -r \sqrt{a^2+\tfrac{2 a (3 K-1) r}{r-1}+\tfrac{(K+1)^2 r^2}{(r-1)^2}}\right)}{4 a (K-1)}+\tfrac{K r+K-2}{4 (K-1)}+\frac{1}{2}
\end{equation*}

The degenerate cases $e^{ik\theta_0}=1$ for $k = 1,2,3$ or $4$ correspond to $\sigma = 1,-1,-\frac{1}{2},0$, so we may determine for which values of $r$ these nondegeneracy conditions are violated. We will solve the equations for $r$, with the constraint that $1<r\leq r_{max}$.

\textit{Case 1: $\sigma = 1$.} This corresponds to $\theta_0=0$, that is 1:1 resonance. There are however no solutions except $r=0$. This means that there is no 1:1 resonance.

\textit{Case 2: $\sigma = -1.$} Then $\theta_0 = \pi$, so this is 1:2 resonance. We find the solution $r=r_{max}$, which means that when $r=r_{max},\beta=\beta_2(=\beta_1)$, there is a 1:2 resonance.

\textit{Case 3: $\sigma = -\frac{1}{2}$.} This is $\theta_0=\frac{3\pi}{2}$, which means 1:3 resonance. We find a solution
$$
\tilde{r} = \frac{3 a+4 K+3}{2 K}+\frac{1}{2} \sqrt{\frac{9 a^2+48 a K-18 a+16 K^2+24 K+9}{K^2}}.
$$
So, for $\beta=\beta_2,r= \tilde{r}$, there is a 1:3 resonance.

\textit{Case 4: $\sigma = 0$.} Then $\theta_0=\frac{\pi}{2}$, corresponding to 1:4 resonance. Here too, there is a solution
$$
\bar{r} = \frac{2 a+3 K+2}{2 K}+ \frac{1}{2} \sqrt{\frac{4 a^2+20 a K-8 a+9 K^2+12 K+4}{K^2}},
$$
which means that for $\beta=\beta_2,r=\bar{r}$ there is 1:4 resonance.

The expressions for $r_{max},\bar{r}$ and $\tilde{r}$ are quite similar, and in fact one can write
\[
\bar{r} = R(2), \  \tilde{r}=R(3),\  r_{max}=R(4),
\]
where
\begin{equation}\label{R(x)}
R(x) =\frac{a x+K (x+1)+x}{2 K}+ \frac{1}{2} \sqrt{\frac{a^2 x^2+2 a x (K (3 x-1)-x)+(K (x+1)+x)^2}{K^2}},
\end{equation}
which we define for $2\leq x \leq4$. In this interval, the derivative of $R$ is
$$
R'(x) = \frac{2 K (a+K+1)+\frac{2 \left(a K (6 x-1)+(a-1)^2 x+K^2 (x+1)+2 K x+K\right)}{\sqrt{\frac{ (ax-(Kx+K+x))^2+8aKx^2}{K^2}}}}{4 K^2}>0
$$
for $2\leq x \leq 4$, and in fact for all positive $x$, which is clear since every term is strictly positive for $x>0$. So $R(x)$ is monotonically increasing for $2\leq x\leq4$, which implies that we always have 
$$
\bar{r}<\tilde{r}<r_{max}.
$$
We should also check that $d\neq 0$, where $d$ is given by \eqref{NS d}. This is quite involved, and in fact we are not able to solve it analytically. 
However, the graph shown in Figure~\ref{d_LimitCycle} shows that $d<0$ for all parameters.  The computation of $d$ is given in appendix \ref{computing d}.
\begin{figure}[!h]
\centering
\includegraphics[width=.5\linewidth]{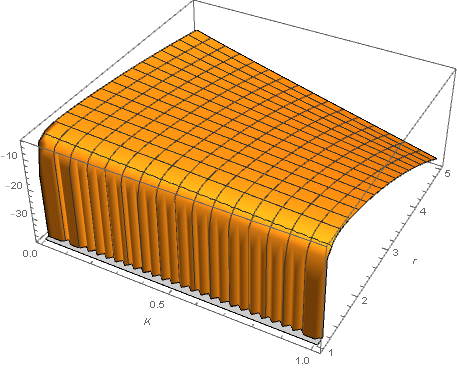}
\caption{Generic investigation of Neimark-Sacker bifurcation}
\label{d_LimitCycle}
\end{figure}
Finally, we check that $\rho'(\beta_2)\neq 0$ where $\rho(\beta) = |\mu_{1,2}(\beta)|$. This is the genericity condition (C.1) given in section 4 in \cite{kuznetsov}. We have that (see appendix \ref{computing d})
$$
\mu_{1,2} = \frac{a_{11}+1\pm i\sqrt{4(a_{11}+Ka_{21})-(a_{11}+1)^2}}{2},
$$
which means that 
\begin{align*}
\rho(\beta) =& |\mu_{1,2}(\beta)| = \frac{1}{2}\sqrt{(a_{11}+1)^2 + 4(a_{11}+Ka_{21})-(a_{11}+1)^2} = \sqrt{a_{11}+Ka_{21}},
\end{align*}
where $a_{11}$ and $a_{21}$ depends on $\beta$. Explicitly
\begin{align*}
a_{11} &= \frac{2 K r}{a K-\beta }+\frac{K (a (r-1)+r)}{\beta }+1\\
a_{21} &= \frac{K (a-(a+1) r)}{\beta }+r-1.
\end{align*}
Hence
\begin{align*}
\rho'(\beta) =& \frac{a_{11}'(\beta )+K a_{21}'(\beta )}{2 \sqrt{a_{11}(\beta )+K a_{21}(\beta )}} = \\
&\frac{-\frac{K^2 (a-(a+1) r)}{\beta ^2}-\frac{K (a (r-1)+r)}{\beta ^2}+\frac{2 K r}{(a K-\beta )^2}}{2 \sqrt{\frac{2 K r}{a K-\beta }+K \left(\frac{K (a-(a+1) r)}{\beta }+r-1\right)+\frac{K (a (r-1)+r)}{\beta }+1}},
\end{align*}
which gives us
$$ 
\rho'(\beta_2) = \frac{(r-1)^2 (2 a (r-1)+(K+1) r)}{2 K \sqrt{2-r} r (a (r-1)+r)}\neq0
$$
for $r>1$.

\section{Bifurcation diagrams and numerical simulations}

To illustrate our results we will in this section provide some numerical simulations and bifurcation diagrams.  
Furthermore we discuss and illustrate existence of period orbit of length \(3\) and possible chaotic behavior supported by computations of the Lyapunov 
exponents.  Fix \(a=1\), \(K=1/2\) we consider two typical \(\beta\)-values: \(\beta =1.1\) and \(\beta=3\) and \(r\) as the bifurcation control parameter.  In Section~3 the bifurcation diagram for \(\beta=1.1\) and below is the bifurcation diagram for \(\beta=3\).   
\begin{figure}[!h]
\centering
\includegraphics[width=.85\linewidth]{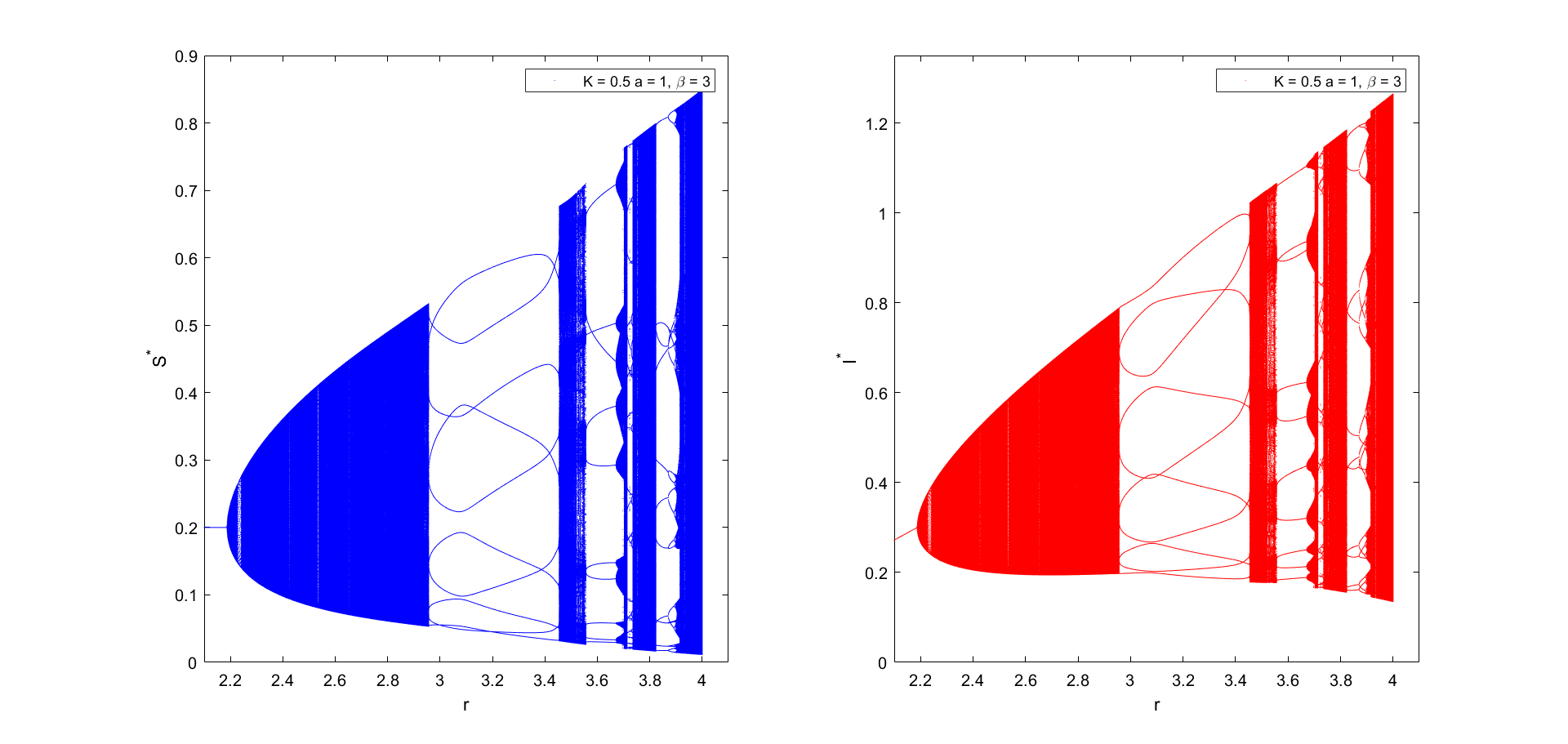}
\caption{Bifurcation diagram with parameter $r$, and $\beta = 3$.}
\label{bifurcation E1r_beta3}
\end{figure}
In this case we have
\begin{itemize}
\item \(1<r<\frac{5}{4}\) corresponding to  \(\beta<\beta_0\): the disease-free fixed point\(E_0\)  is stable;
\item \(\frac{5}{4}<r<\frac{35}{16}\approx2.183\), i.e. \(\beta_1<\beta<\beta_2\): the endemic fixed point \(E_0\)  is stable;
\item  \(r>2.183\) Neimark-Sacker's limit cycles are stable, dipicted in Figure~\ref{r22-cycles} and Figure~\ref{r25-cycles}.
\begin{figure}[!h]
\centering
\subfloat{\includegraphics[width=.33\linewidth]{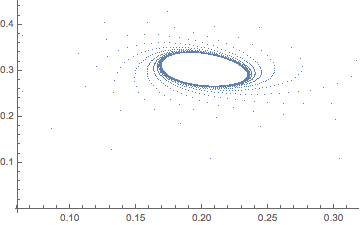}}
\subfloat{\includegraphics[width=.33\linewidth]{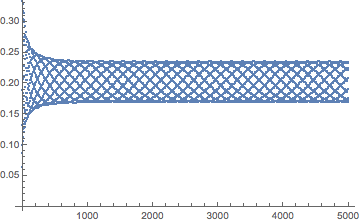}}
\subfloat{\includegraphics[width=.33\linewidth]{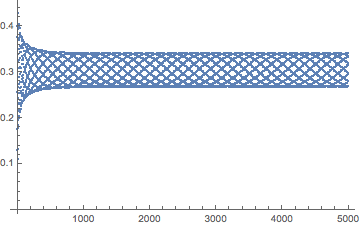}}
\caption{Phase portrait and trajectory of Neimark-Sacker's Limit cycle  $r=2.2$.}
\label{r22-cycles}
\end{figure}

\begin{figure}[!h]
\centering
\subfloat{\includegraphics[width=.33\linewidth]{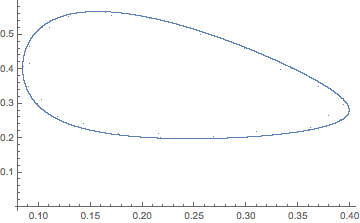}}
\subfloat{\includegraphics[width=.33\linewidth]{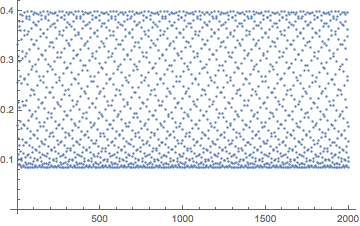}}
\subfloat{\includegraphics[width=.33\linewidth]{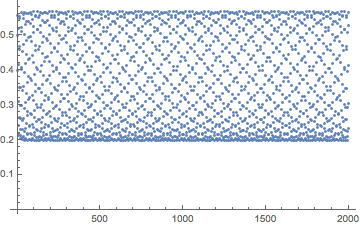}}
\caption{Phase portrait and trajectory of Neimark-Sacker's Limit cycle  $r=2.5$}
\label{r25-cycles}
\end{figure}
\item As \(r\) passes \(3\) the trajectories accumulated on a cycle but with clear pattern of \(n\)-cycles or \(n\)-limit cycles shown in Figure~\ref{r33-cycles} and  Figure~\ref{10_LC}.
\begin{figure}[!h]
\centering
\subfloat{\includegraphics[width=.33\linewidth]{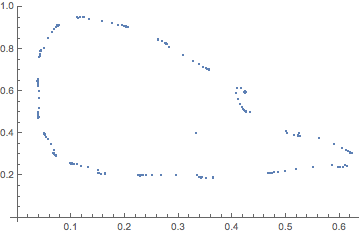}}
\subfloat{\includegraphics[width=.33\linewidth]{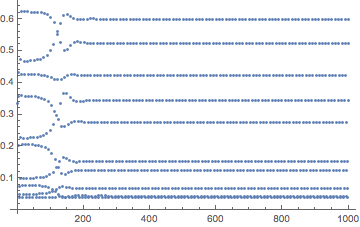}}
\subfloat{\includegraphics[width=.33\linewidth]{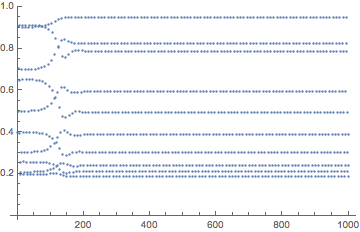}}
\caption{A10-cycle on the invariant set "Neimark-Sacker's Limit cycle"  $r=3.3$.}
\label{r33-cycles}
\end{figure}
\begin{figure}[!h]
\centering
\subfloat{\includegraphics[width=.33\linewidth]{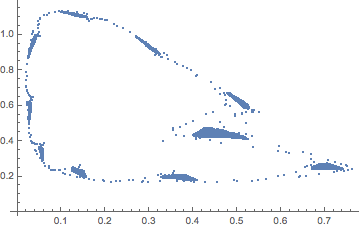}}
\subfloat{\includegraphics[width=.33\linewidth]{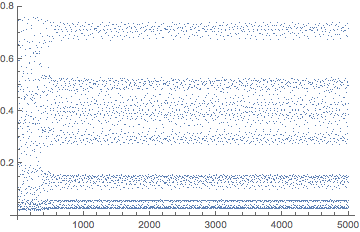}}
\subfloat{\includegraphics[width=.33\linewidth]{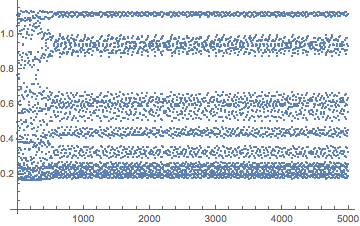}}
\caption{10 small cycles on the invariant set "Neimark-Sacker's Limit cycle"  $r=3.3$.}
\label{10_LC}
\end{figure}
\end{itemize}
This indicates a further bifurcation.  In this case the higher order in approximation should be taken into consideration.  Generically there is only finite number of periodic orbits on the closed invariant curve as (\cite{kuznetsov}).   

Next we show a bifurcation diagram with \(\beta\) as a control parameter for \(r=3.6\), \(a=1\) and \(K=0.5\), shown in Figure \ref{Bif_beta_K05_a1_r36} 
 and zoomed in for larger \(\beta\)-values in Figure\ref{Bif_beta_K05_a1_r36_short}.
 \begin{figure}[!h]
\centering
\includegraphics[width=1\linewidth]{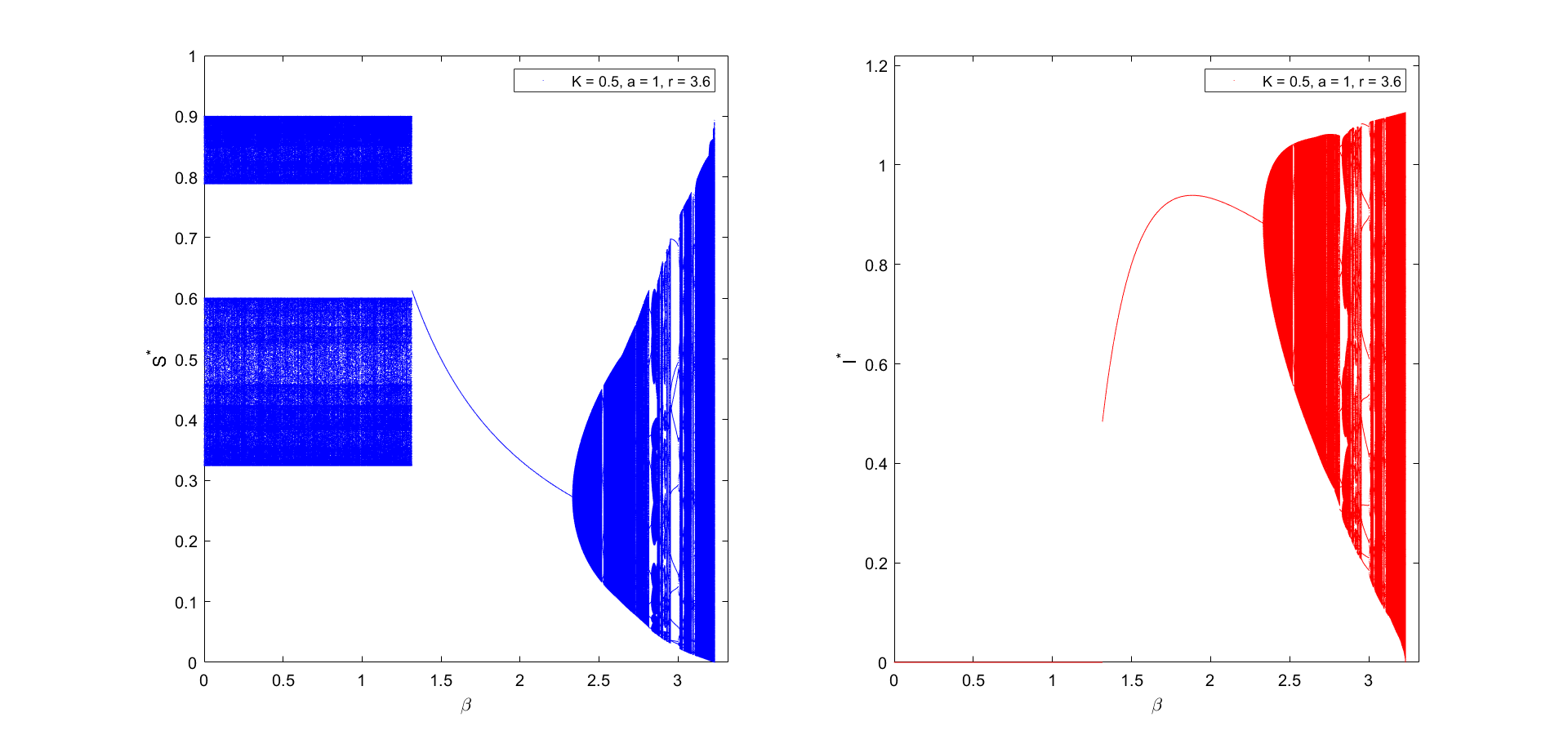}
\caption{Bifurcation diagram with control parameter \(\beta\), where \(r=3.6,K=.5,a=1\)}
\label{Bif_beta_K05_a1_r36}
\end{figure}
\begin{figure}[!h]
\centering
\includegraphics[width=1\linewidth]{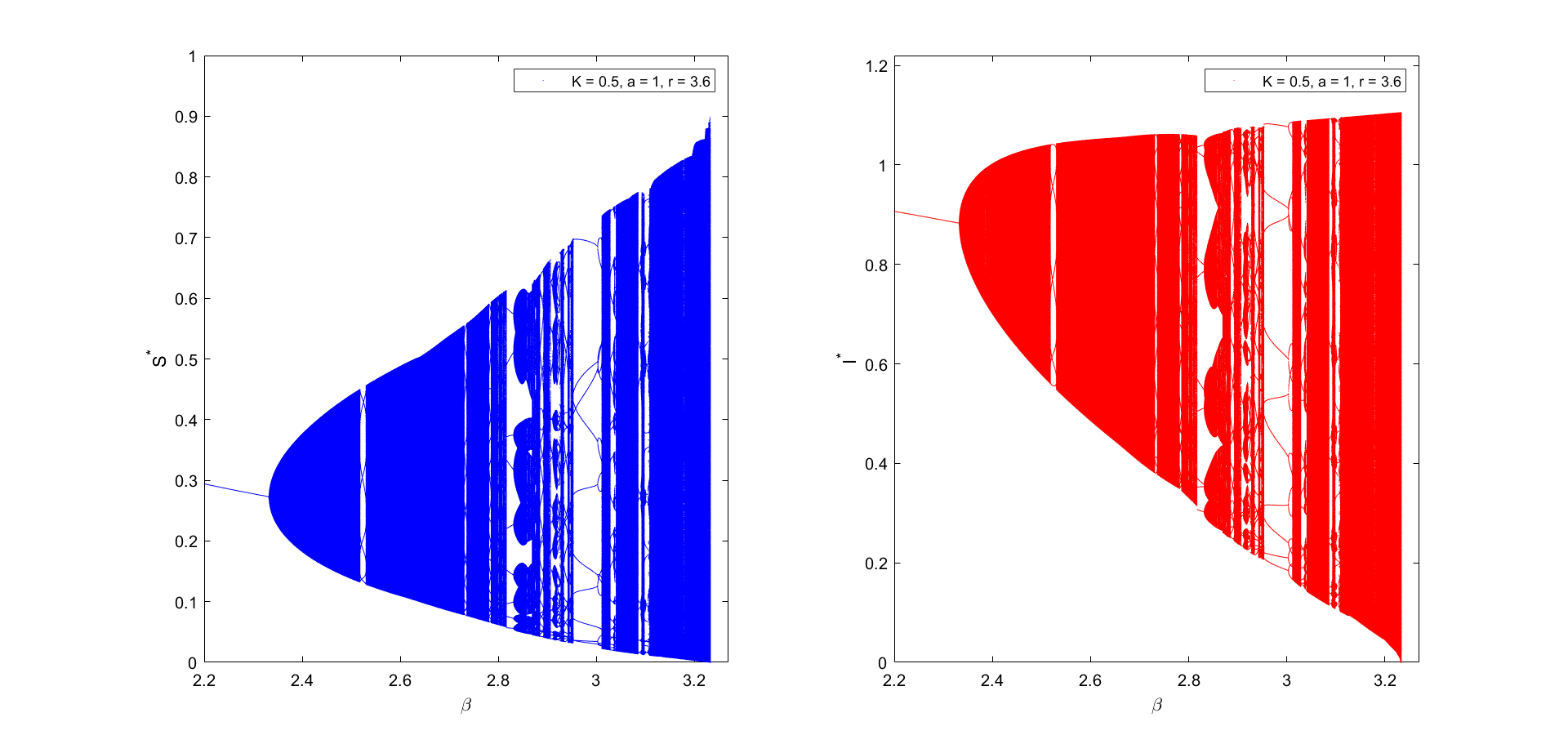}
\caption{Bifurcation diagram with control parameter \(\beta\), where \(r=3.6,K=.5,a=1\)}
\label{Bif_beta_K05_a1_r36_short}
\end{figure}
When \(1.31478<\beta <3.33251\), the trajectories goes to \(E_1\).  Beyond \(\beta_2\) there is Neimark-Shcker's limit cycle and \(n\)-cycles on a closed curve  and eventually chaos.  Note that when \(\beta<1.31478\) there is chaos in for most \(\beta\)-values which will be discussed later.

We close this subsection by showing phase portrait and \(S\)- and \(I\)-orbits for parameters \(\beta=2.33\) corresponding to convergence to \(E_1\), Figure~\ref{beta23r36-E1}, and a phase portrait on a bifurcation from Neimark-Sacker's limit cycle with 10 small cycles on a closed invariant curve, Figure~\ref{LC_LC}.

\begin{figure}[!h]
\centering
\subfloat{\includegraphics[width=.33\linewidth]{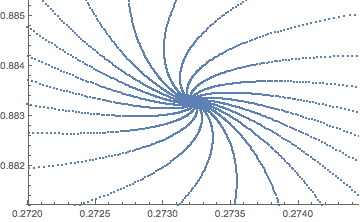}}
\subfloat{\includegraphics[width=.33\linewidth]{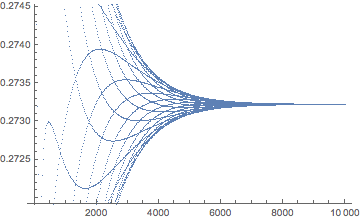}}
\subfloat{\includegraphics[width=.33\linewidth]{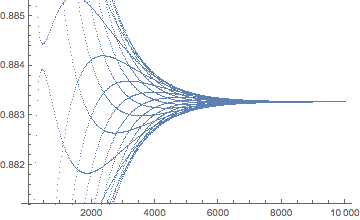}}
\caption{Convergence to \(E_1\)  $\beta=2$.}
\label{beta23r36-E1}
\end{figure}

\begin{figure}[!h]
\centering
\includegraphics[width=0.4\linewidth]{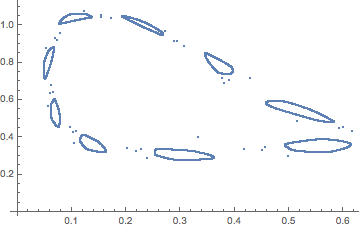}
\caption{A further bifurcation from Neimark-Sacker's limit cycle: \(\beta=2.85\) \(r=3.6,K=.5,a=1\)}
\label{LC_LC}
\end{figure}

\subsection{Existence of three-cylcles}

The next simplest type of orbit is a cycle. In discrete-time systems, a cycle of length $k$ corresponds to a fixed point of the $k$:th iterate $f^k$.  We showed that there is period doubling  when \(3<r<4\) and \(\beta<\beta_0\).  An interesting question to pose is whether one can draw any conclusions about the existence of cycles of other lengths from the presence of a cycle of length $k$. 

In the paper 'Period three implies chaos' \cite{Li}, Li and Yorke were the first to introduce the term \textit{chaos} in mathematics. In the paper, they show that if a continuous map has a cycle of period 3, then it must have cycles of any period $k$. This quite non-intuitive result is in fact a special case of a remarkable theorem of Sharkovskii. To state the theorem, we must first present a new ordering $\triangleright$ of the positive integers as follows:
\begin{align*}
&3\triangleright5\triangleright7\triangleright\dots\triangleright2\cdot3\triangleright2\cdot5\triangleright2\cdot7
\triangleright\dots\triangleright2^2\cdot3\triangleright2^2\cdot5\triangleright2^2\cdot7\dots
\triangleright\dots\\
&\triangleright2^n\cdot3\triangleright2^n\cdot5\triangleright2^n\cdot7\dots
\triangleright
\dots\triangleright2^n\triangleright2^{n-1}\triangleright\dots\triangleright2^2\triangleright2\triangleright1.
\end{align*}
First the odd integers are listed, except 1, then 2 times the odd integers, followed by $2^2$ times the odd integers, and in general $2^n$ times the odd integers for all positive integers $n$. Finally, one lists the powers of 2 in descending order. Clearly all positive integers are generated this way. The notation $m \triangleright n$ means that the positive integer $m$ comes before $n$ in the Sharkovskii ordering. In particular, this means that $3\triangleright k$ for any positive integer $k$.   More precisely 

\begin{theorem}
Let $f:I\to I$ be a continuous map on the interval $I$, where $I$ may be finite, infinite, or the whole real line. If $f$ has a cycle of period $k$, then it has a cycle of period $r$ for all $r$ with $k \triangleright r$. 
\end{theorem}

Hence it is of interest whether the system under consideration has a 3-cycle.   To investigate  this phenomenon we search for the parameters based on the fact that  the system \eqref{syst} simply becomes the logistic map when $I = 0$ and the flip bifurcation from \(E_0\) (\(r=3\)) is stable, we could expect that the \(S\)-trajectories preserve properties of the logistic map that has has a 3-cycle, for then our system would also inherit this cycle when $I$ tends to  $0$.

Following \cite{saha},  $r = 1+2\sqrt{2}$. results in a 3-cycle, shown in Figure \ref{3-cycle sim and bif}, the orbit of $(S,I)$ stabilizes after about 30 iterations to a 3-cycle. 

\begin{figure}[!h]
\centering
\subfloat{\includegraphics[width=.5\linewidth]{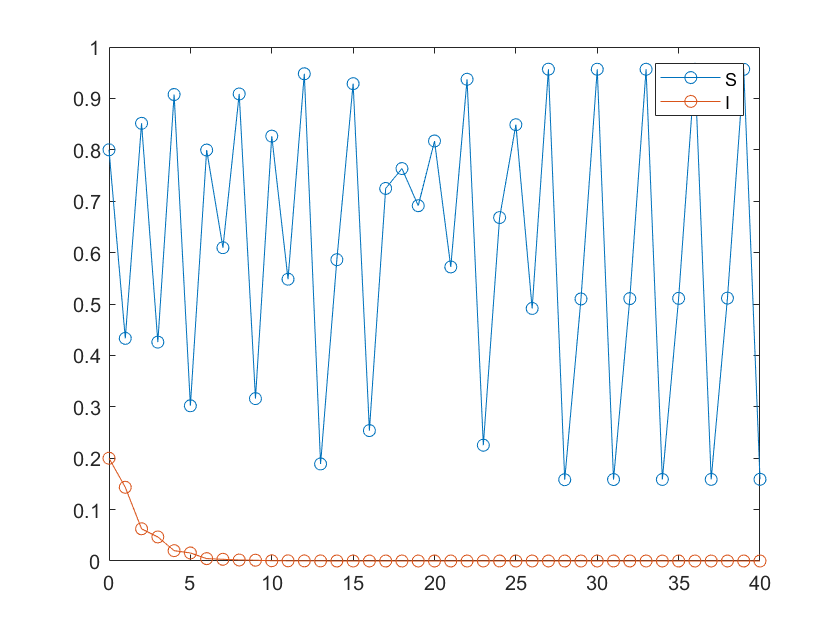}}
\subfloat{\includegraphics[width=.5\linewidth]{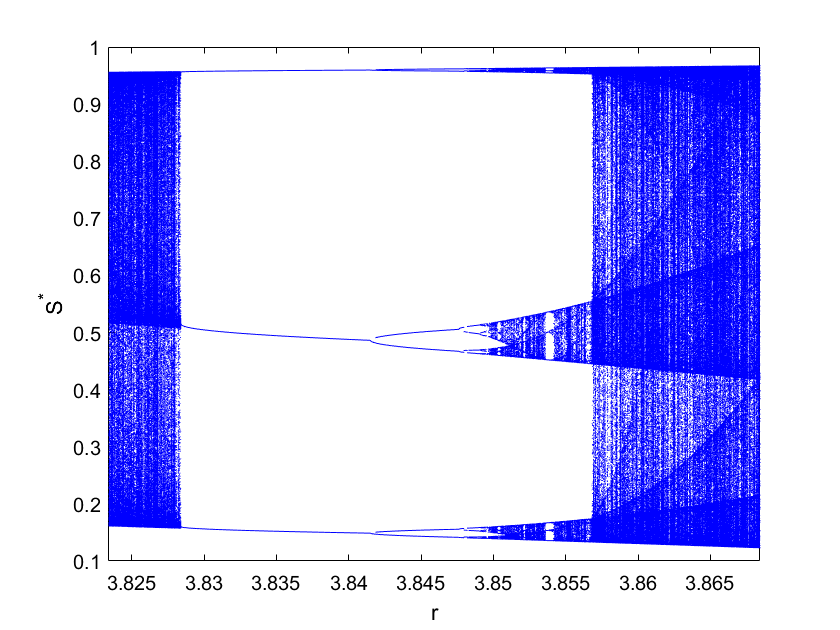}}
\caption{Time series for $r = 1+2\sqrt{2},\beta = 2.09,\mu=0.19,\gamma = 0.99,a=1.08$, and $S_0 = 0.8, I_0 = 0.2$, and part of the bifurcation diagram for these parameters as $r$ varies in a neighbourhood of the critical value.}
\label{3-cycle sim and bif}
\end{figure}

\subsection*{Some $n$-cycles}

Now that we know there is a 3-cycle, Sharkovskii's theorem tells us that there are cycles of arbitrary length. We can solve the system specified in \cite{saha} for $n=5$ numerically which yield three distinct solutions greater than 3, namely $r_1=3.73817,r_2 = 3.90557,r_3=3.99026$. We expect these values of $r$ to yield 5-cycles in the bifurcation diagram when $\beta<\beta_0$, and indeed Figure \ref{all 5-cycles} show all three of them.

\begin{figure}
\centering
\subfloat{\includegraphics[width=.33\linewidth]{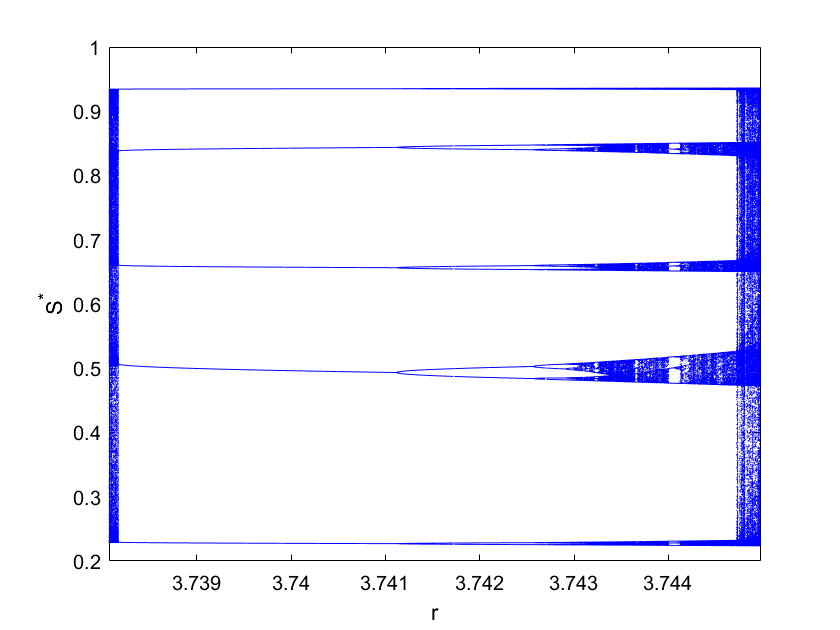}}
\subfloat{\includegraphics[width=.33\linewidth]{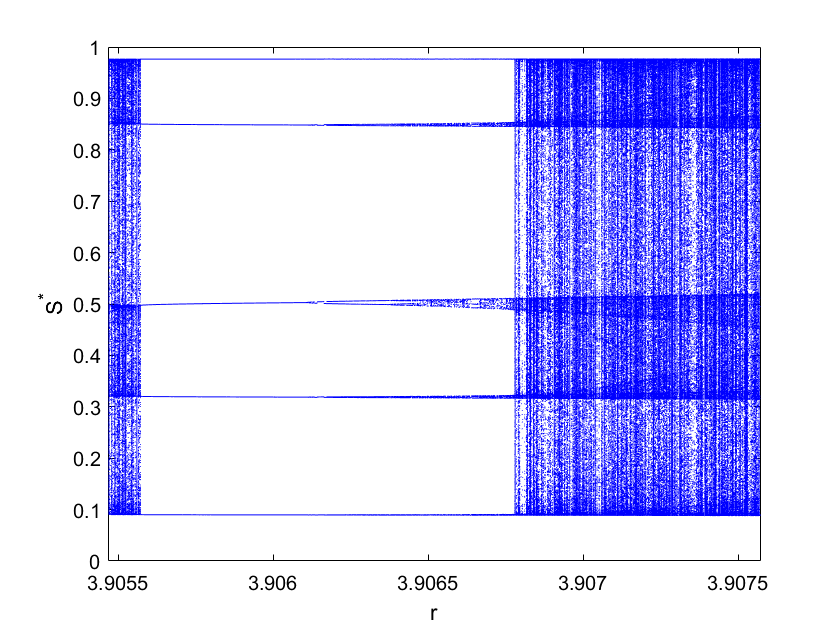}}
\subfloat{\includegraphics[width=.33\linewidth]{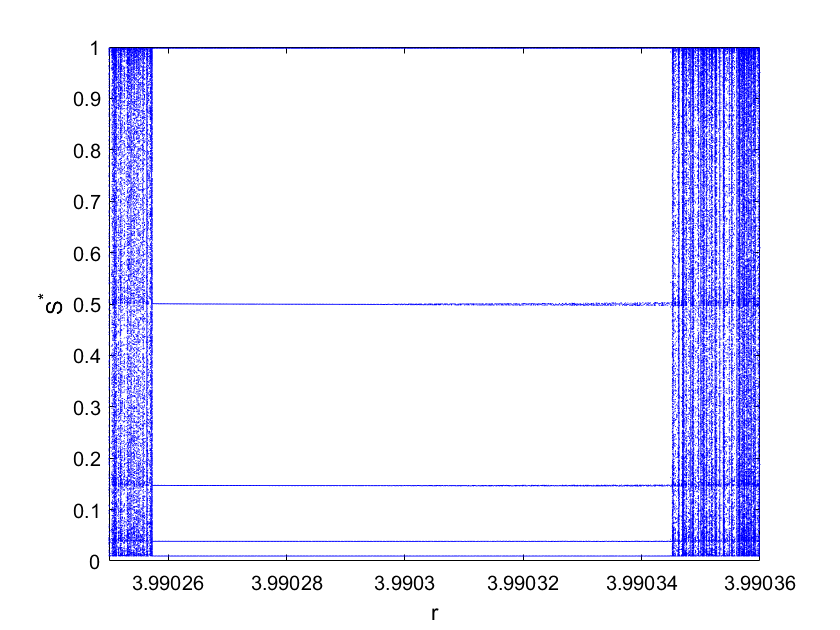}}
\caption{All the 5-cycles found in the bifurcation diagram as $r$ varies, and with $\beta<\beta_0$.}
\label{all 5-cycles}
\end{figure}

For $n=6$ we can solve the system of equations numerically, and find eight values of $r$ that are greater than 3, namely 
$$
r=3.21486, 3.63386,3.83185,3.83265,3.85556,3.93769,3.97781,3.99759,
$$
and with patience one can numerically find all nine solutions greater than 3 when $n=7$. For completeness these are
$$
r = 3.71955,3.78707,3.88935,3.92373,3.95204,3.96955,3.98497,3.99461,3.99941.
$$

For larger $n$ it is no longer practical to solve the system of equations. We can however by simply looking at the bifurcation diagram find some more cycles. As an example, Figure \ref{n-cycles} show a 7-cycle, a 10-cycle and an 18-cycle.

\begin{figure}[!]
\centering
\subfloat{\includegraphics[width=.33\linewidth]{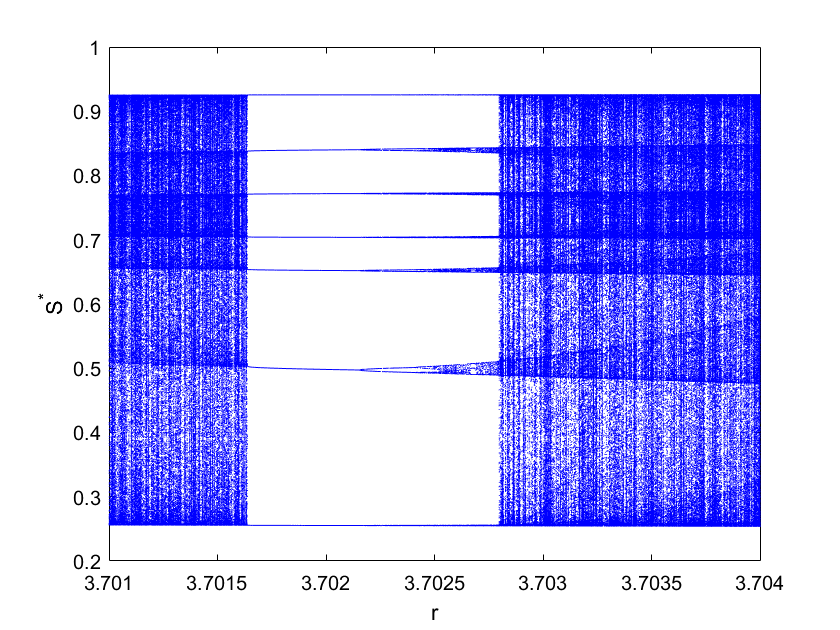}}
\subfloat{\includegraphics[width=.33\linewidth]{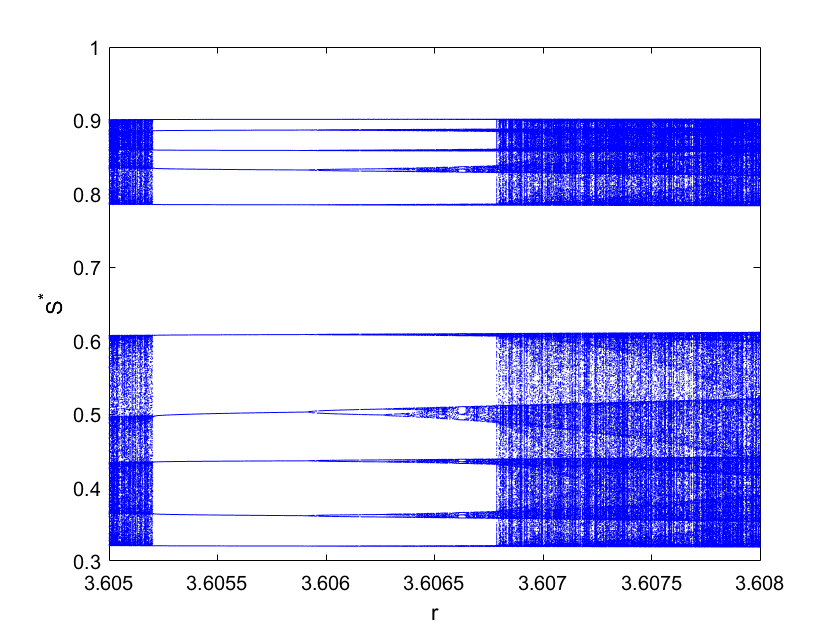}}
\subfloat{\includegraphics[width=.33\linewidth]{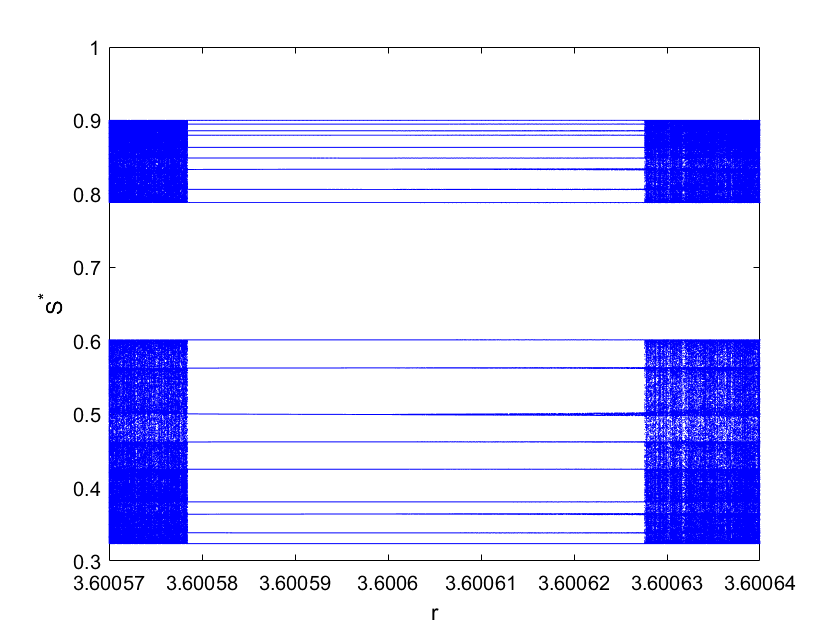}}
\caption{Cycles of length 7, 10 and 18, when $\beta<\beta_0$.}
\label{n-cycles}
\end{figure}

\subsection{Chaotic behaviour and Lyapunov exponents}

In Section~4 we argued that our model behaves like the logistic mapping if \(\beta<\beta_0\) and other \(n\)-cycles are also shown in the previous subsection.  
Now we argue that Period \(3\) indeed leads to chaos.  To this end we compute the Lyapunov exponents \cite{LE} for \(3.5<r<4\), see Figure~
\ref{LE} (left).  Clearly we have the smallest Lyapunov exponent greater than \(0\) in this range of \(r\) and it agrees to the stable \(n\)-cycles we found 
previously when  the largest Lyapunov exponent is less than \(0\)

Now we turn the second bifurcation diagram for \(a=1\), \(K=0.5\) and \(\beta=3\).  We plot the Lyapunov exponents in  Figure~\ref{LE} (right) for \(3.45<r<4.18\).  Again it agrees to the discussion above for the limit cycles and other periodic orbits.    And similarly  a Lyapunov exponents for the diagram in Figure~\ref{Bif_beta_K05_a1_r36_short} is provided.

\begin{figure}[!h]
\centering
\includegraphics[width=1\linewidth]{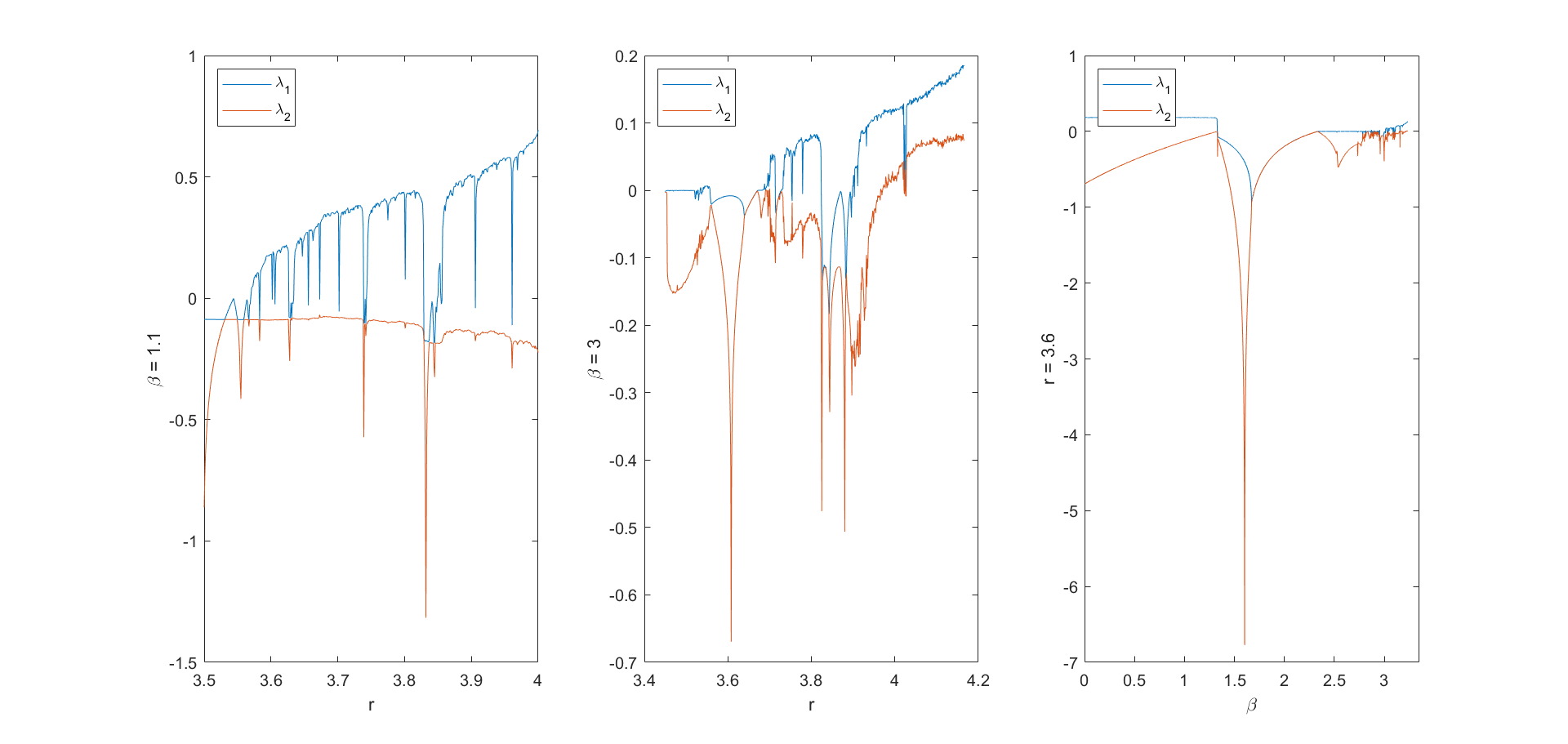}
\caption{Lyapunov exponents: \(\beta=1.1\) (left), \(\beta=3\) (middle),\(r=3.6\)}
\label{LE}
\end{figure}

Figure~\ref{chaos}  shows a phase plot and the trajectory \(S_n\) where  \(S_0\) is \(1/3\) and its round off \(0.3333\).  
\begin{figure}[!h]
\centering
\subfloat{\includegraphics[width=.33\linewidth]{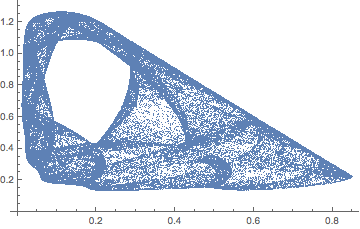}}
\subfloat{\includegraphics[width=.33\linewidth]{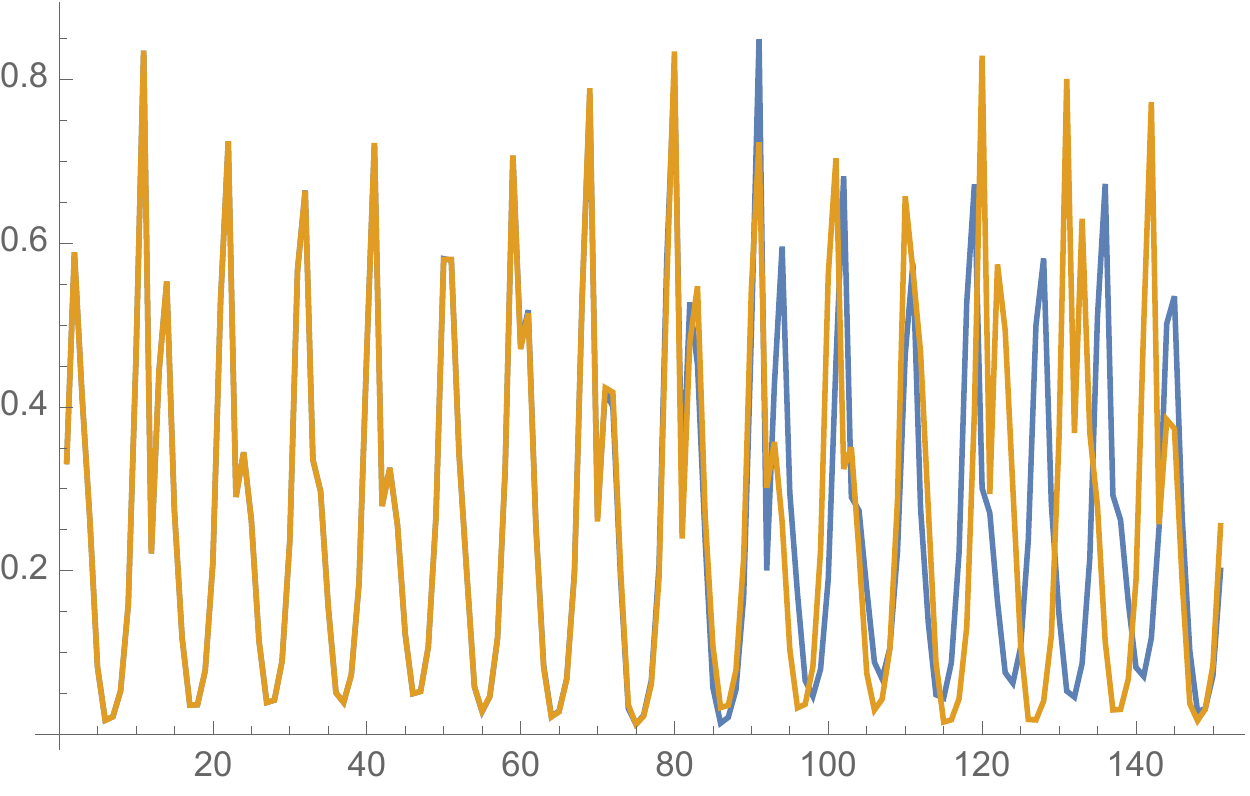}}
\caption{Phase portrait on \(SI\)-plane and \(S\)-trajectory starting at \(1/3\) and \(0.3333\), respectively.}
\label{chaos}
\end{figure}
Clearly the trajectories are very different after some steps even though the round off is not very rough.  
Note also that the phase plot shows different density of the points and the shape of it is much like the cycle case for \(r=3.3\) which may be compared with the situation in case \(\beta<\beta_0\), that is, stable  \(n\)-cycles in chaos.  We think that \(r\) could get larger than \(4\) in case \(\beta=3\) is due to the fact that  a positive term is substracted from the logistic mapping in the \(S\)-iteration.  Beyond that almost all initial values will result an unbounded trajectory except that initial values are extremely close to a stable fixed point.

Finally in the last bifurcation diagram it shows that the system exhibits chaotic behavior starting with \(\beta\) approximately \(3\), which can be inspected by the Lyapunov exponents Figure~\ref{LE}.  

\section{Discussions and further remarks}

In this section we will discuss several issues from epidemiological point of view and point out issues remained unanswered.

\subsection{Basic reproduction number and disease control}

A central number in epidemiology is the so-called {\em basic reproduction number}, denoted $\mathcal{R}_0$. 
It is defined as the expected number of secondary cases produced by a single (typical) infection in a completely susceptible population. It is important 
to note that \(\mathcal{R}_0\) is a dimensionless number and not a rate, which would have units of time \cite{JonesR0}.
The use of this quantity is not without complications, but as a rule of 
thumb, one says that if $\mathcal{R}_0<1$ the infection dies out in the long run, and if $\mathcal{R}_0>1$ the infection will spread in the population and will require intervention to eradicate.  Nevertheless we can turn the question around to ask under what conditions the disease will die out then the threshold is the number
\begin{equation}\label{R_0 cand1}
	\frac{\beta(r-1)}{K(r+a(r-1))}=\dfrac{\beta}{K\left(a+1+\frac{1}{r-1}\right)}
\end{equation}
This number being less than \(1\) is equivalent to $\beta < \beta_0$, which implies that $E_0$ is (locally) stable and thus $I_n\to0$ as $n \to\infty$, at least when $r\le3$. 
The bifurcation diagram in Fig. \ref{bifurcation E0r} suggests that this remains the case for $3\leq r\leq 4$. 
From this number we can see that given \(\beta\) we have three possibilities to make this number smaller; making \(r\) smaller (in other words, to control the population growth and supplies) or making \(K\) large (which is to treat the infectives efficiently so that they are removed) or make \(a\) large (the measure imposed on the community strong enough).  In this sense we provide an even simpler number that guarantees (\ref{R_0 cand1})
\begin{equation}\label{R_0 cand2}
\frac{\beta}{(a+1)K}.
\end{equation}

%%%%%%%%%%%%%%%%%%%%%%%%%%%%%%%%
We show that if this number is less than 1,  $I_n\to0$ for all biologically permitted initial values as long as $r\leq4$ (by lemma \ref{f<1}). To this end let \(h_1(x,y):= r(1-x)-\frac{\beta y}{1+ax}\), \(h_2(x):=\frac{\beta x}{1+ax}+1-K\).  Then $f(S,I):=rS(1-S)-\frac{\beta SI}{1+aS}=S h_1(S,I)$ and $g(S,I) :=(1-K)I+\frac{\beta SI}{1+aS}I h_2(S)$.
Since \(h_2'(x)=\frac{\beta}{(1+ax)^2}>0\), \(h_2\)  is strictly increasing for all real \(x\). 

\begin{proposition}  \label{p1-p}
The following bounds on $h_2(x)$ hold:
\begin{enumerate}
\item If \(\beta<K(1+a)\), then \(0<h_2(x)<1\) for all \(0<x<1\).
\item  If \(\beta\ge K(1+a)\), then \(0<h_2(x)<1\) for all \(0<x<\frac{K}{\beta-aK}\); and \(1<h_2(x)<\frac{\beta}{1+a}+1-K\) for \(\frac{K}{\beta-aK}\le x< 1\).

\end{enumerate}
\end{proposition} 
\begin{proof}  By strictly increasing property of \(h_2\), we have, if \(\beta<K(1+a)\)
\[
h_2(0)<h_2(x)<h_2(1)=\frac{\beta}{1+a}+1-K<1.
\]
proving the first assertion.  Note that \(\beta>K(1+a)\) is equivalent to \(\frac{K}{\beta-aK}<1\).  Then \(h_2\) is strictly increasing implies that 
\[
h_2(x)< h_2\left(\frac{K}{\beta-aK}\right)=1, 
\]
if  \(0<x<\frac{K}{\beta-aK}\).  Similarly
\[
1=h_2\left(\frac{K}{\beta-aK}\right)<h_2(1)=\frac{\beta}{1+a}+1-K
\]
if  \(\frac{K}{\beta-aK}<x<1\).
\end{proof}

Consequently,  if \(\beta<K(1+a)\), \(g(x,y)<y\) for all \(0<x<1\); and \(\beta\ge K(1+a)\),  \(g(x,y)<y\) for all \(0<x<\frac{K}{\beta-aK}\). 

\begin{proposition}   \label{p3-p}
If \(\beta<(1+a)K\) then \(0<I_{n+1}<I_n<I_0\).   Furthermore, 
\[
\gamma_1^n I_0\le I_n\le\gamma_2^n I_0,  \ \text{ for all \(n\ge1\)}
\]
where \(\gamma_1:=1-K\) and \(\gamma_2:=\frac{\beta}{1+a}+1-K\).
\end{proposition}

\begin{proof}   We have  \(0<I_{n+1}=h_2(S_n)I_n<I_n<I_0\) for all \(0<S_n<1\) by Proposition~\ref{p1-p}.   This shows \(I_n\) is monotonically decreasing and bounded by \(0\) below and  \(I_0\) above.    By the same proposition, we obtain the estimates 
\((1-K)I_n\le I_n\le\gamma_2 I_n\) for all \(n\ge1\).  Repeating these inequalities yields the desired estimates.
\end{proof}

Note that \(\gamma_1=1-K<1\) and \(\gamma_2<1\) due to  \(\beta<K(1+a)\), \(I_n\to0\) as \(n\to\infty\) for all \(S_n\in(0,1)\). Thus if the number \eqref{R_0 cand2} is less than 1,  $I_n\to0$ as claimed. The number \eqref{R_0 cand2} has the virtue of being simpler, and it provides some nice insight as we essentially compare the force of infection $\beta$ to all the factors that prevents spread, namely $K$ the combined death and recovery rate, and $a$ the measures taken to prevent spread among the susceptible population.  We leave thus the further investigation of basic reproduction number to 
epidemiological study.

As the model was given, there is just one control parameter, $a$ which is said to be the inhibitory effect. This parameter acts only on the susceptible part of the population so it could be isolation, vaccination or something along those lines. In Fig. \ref{Bifa} we show a bifurcation diagram that show it is possible to go from a situation where $E_1$ is stable to one where $E_0$ is stable by increasing $a$.

As pointed out above there is  another way to decrease the spread of the disease, namely by increasing $K$, the sum of deaths due to disease and recovery. There are of course two ways to do this, and they are mathematically symmetrical, but  the nicer possibility is to cure the infected so as to inhibit further spread. In Fig. \ref{BifK} we see that it is indeed possible to go from stable $E_1$ to stable $E_0$ by increasing $K$. 
\begin{figure}[h!]
\centering
\includegraphics[width=.9\linewidth]{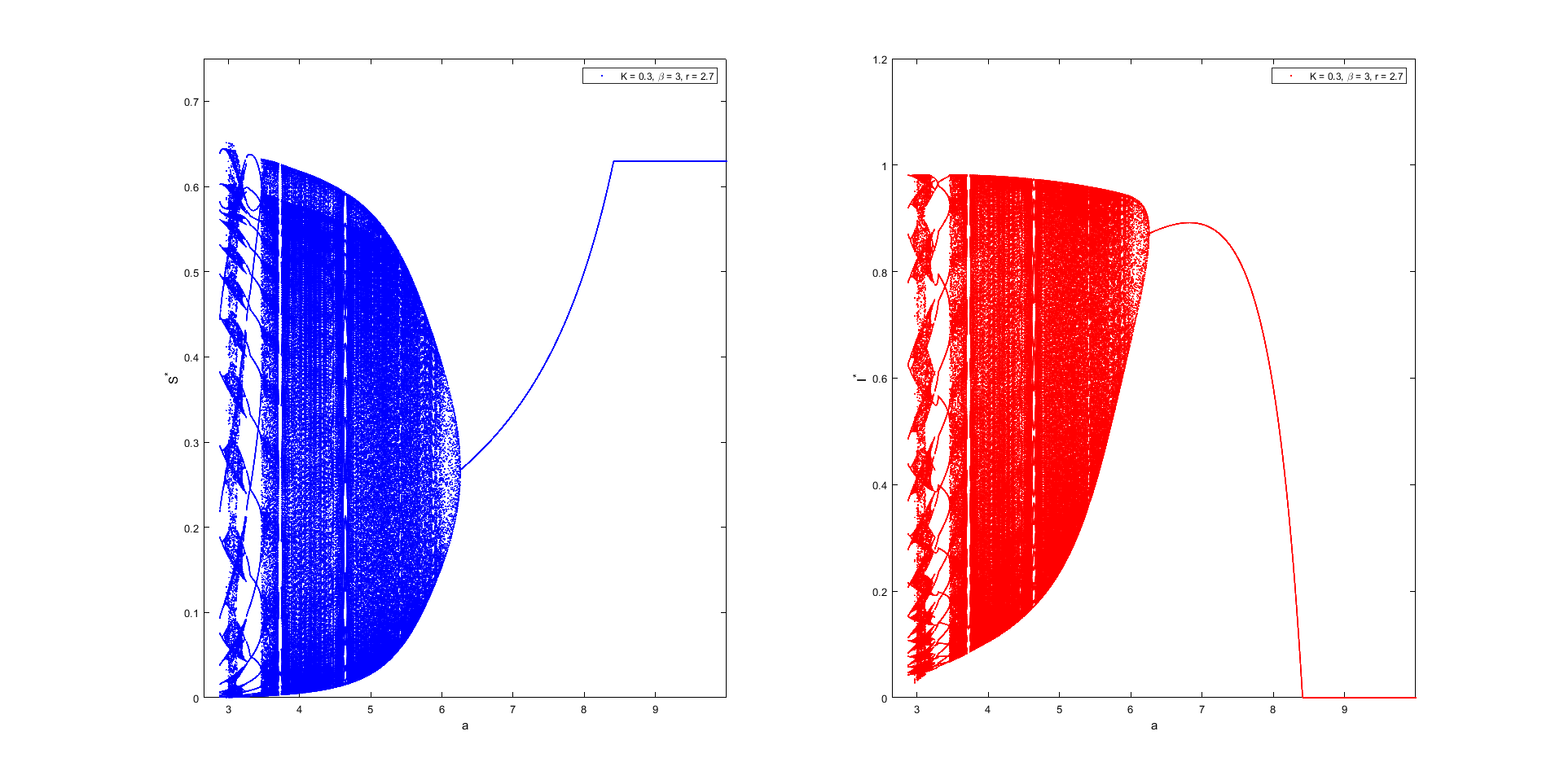}
\caption{Bifurcation diagram with parameter $a$.}
\label{Bifa}
\end{figure}
\begin{figure}[h!]
\centering
\includegraphics[width=.9\linewidth]{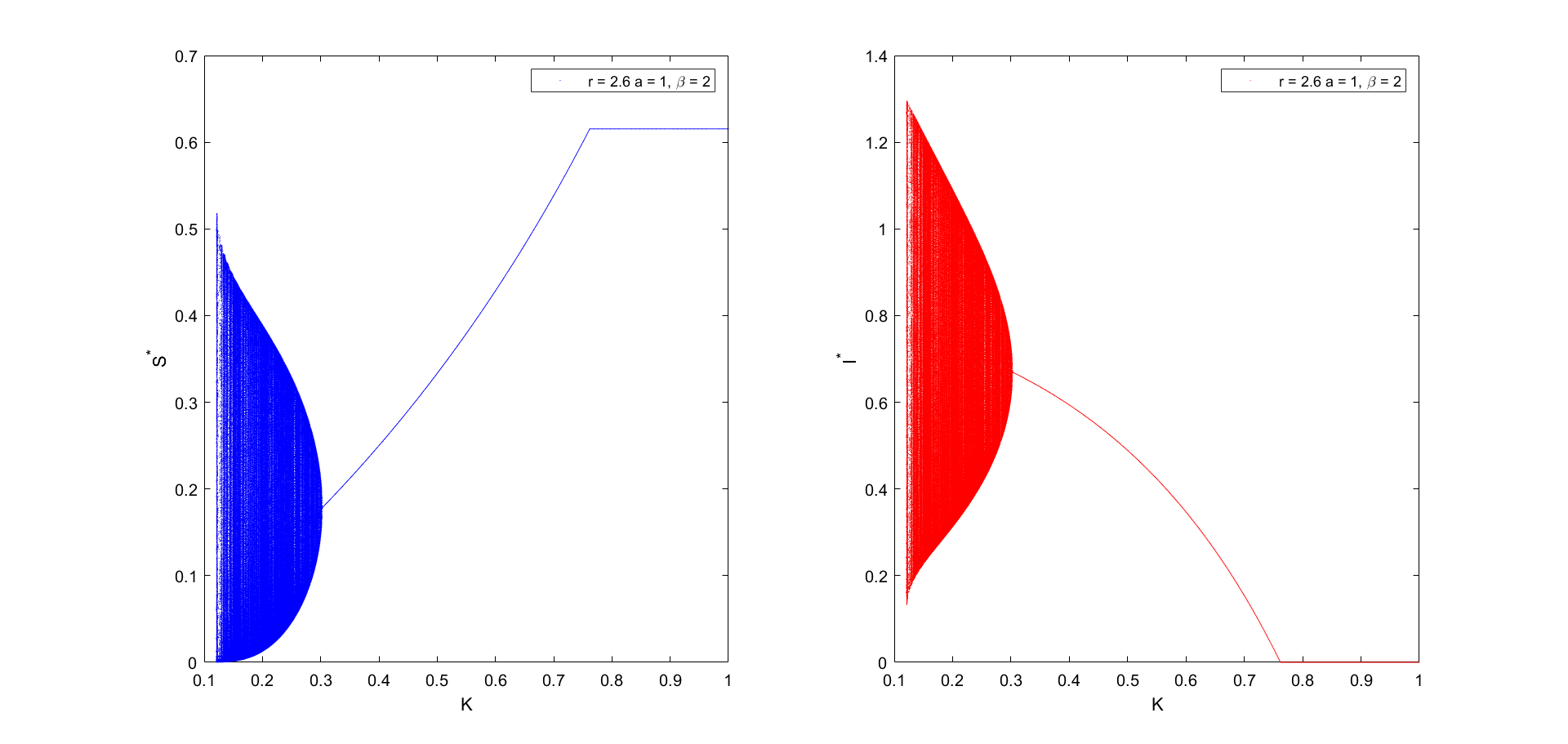}
\caption{Bifurcation diagram with parameter $K$.}
\label{BifK}
\end{figure}

%We have assumed throughout that $0<K<1$. It is however clearly possible to have $K = 1$ but not beyond. This corresponds to a situation where all infected either dies or recovers in one unit of time. We provide bifurcation diagrams along with some simulation in this special case as well in Appendix~C.

\subsection{Further investigations}

We think that the following issues are worth further investigation.  From point of view of dynamical system theory and epidemiology it is desirable to have  
global convergence to a fixed point, in our case 
to have a description on a larger positive region of  \({\mathbb R}^2\) from which the iteration start will converge to the endemic fixed point \(E_1\). 
%The second important issue is to determine a (positive)  invariant set.  For example  our numerical examples show that there is a donut-like set with \(E_1\) in the middle and or the Neimark-Sacker's limit cycle with \(E_1\) inside.  

In section 2 we presented conditions for positive trajectories in proposition \ref{Omega positive}. As pointed out this does not cover all possible cases, so some further investigation is needed to give a complete picture of the positive trajectories. 

We demonstrated  chaotic dynamical behaviour for some parameters numerically.  Obviously it will be nice to have a proof on its existence.  
Similarly, a complete analysis on the bifurcation from Neimark-Sacker's limit cycle to \(p\)-periodic cycles/orbits on an invariant closed curve shown in Figure~\ref{10_LC} and Figure~\ref{LC_LC} needs a rigorous  mathematical analysis.

In order to make the current study and the related results have more  impacts in epidemiology, the model might be modified.  One potential approach is to use Ricker-type models:
\[
\begin{aligned}
S_{n+1}&=S_ne^{r(1-\frac{S_n}{K})-\frac{\beta I_n}{1+ aS_n}}\\
I_{n+1}&=S_ne^{r(1-\frac{S_n}{K})}(1-e^{-\frac{\beta I_n}{1-a S_n}})e^{-\gamma-\mu}\\
R_{n+1}&=S_ne^{r(1-\frac{S_n}{K})}(1-e^{-\frac{\beta I_n}{1-a S_n}})(1-e^{-\gamma})e^{-\mu}.
\end{aligned}
\]
The analysis presented in this paper can provide some insights into this new system.

%\bibliographystyle{tfs}
%\bibliography{references} 

\appendix

%
%\section*{Appendices}
%\appendix

\section{Proof of Proposition \ref{Omega positive}}\label{proof}

We proof Proposition \ref{Omega positive} by proving a series lemmas.  Through out this section we assume \(0<K<1\).

\begin{lemma}\label{f<u^*}
If $\sqrt{K}+1\leq r \leq(\sqrt{K}+1)^2$, then $f(x,u)\leq u^*$
\end{lemma}
\begin{proof}
By Lemma \ref{f<1} we have $f(x,y)<1$ so we require that $\frac{r}{4}\leq u^*\leq1$. Now $u^* = \frac{(r-1+K)^2}{4Kr}\leq1$ if and only if $(r-1+K)^2\leq 4Kr$ which is equivalent to $(r-1-K-2\sqrt{K})(r-1-K+2\sqrt{K})\leq0$. 
This is true if and only if $(\sqrt{K}-1)^2\leq r \leq (\sqrt{K}+1)^2\leq 3$. Since $0<K<1$ both the first and last inequalities are automatically true. Hence we have $r\leq (\sqrt{K}+1)^2$.

Moreover $\frac{r}{4}\leq u^*$ if and only if $Kr^2\leq(r-1+K)^2$ which is $\sqrt{K}r = \sqrt{K} + (r-1)\sqrt{K} \leq r-1+K$. Rearranging slightly yields $\sqrt{K}(1-\sqrt{K}) \leq (r-1)(1-\sqrt{K})$, i.e $\sqrt{K}+1\leq r$ since $0<K<1$.

Summing up, if $\sqrt{K}+1\leq r \leq(\sqrt{K}+1)^2$, then $f(x,u)\leq u^*$ as stated.
\end{proof}
\begin{lemma}\label{f>0}
If $(x,y)\in\Omega_1$ and either $\beta<r$ or $r<\beta<\frac{4Kr^2}{(r-1+K)^2}$, then $f(x,y)\geq0$.
\end{lemma}
\begin{proof}
Note that $f(x,y)\geq0$ if and only if $h(x,y) = r(1-x)(1+ax)-\beta y\geq0$ for positive $x$ so it is sufficient to determine under which conditions $\min_{\Omega_1}h(x,y)\geq0$. 
Now since $h$ is a concave function on the polytope $\Omega_1$ so the minimum is attained at some vertex \cite{polytope}. 
The vertices are $(0,0),(u^*,0)$ and $(0,u^*)$ and clearly $h(x,0)\geq h(x,y)$ for all $(x,y)\in \Omega_1$. Thus $\min_{\Omega}h(x,y) = h(0,u^*) = r-\beta u^*$. This is non-negative if and only if $u^*\leq\frac{r}{\beta}$.
We have two cases to consider. If $\frac{r}{\beta} >1$ we requre nothing more since $u^*\leq 1$. We can simply note that $r>\beta$. 
If however $\frac{r}{\beta} <1$ we require additionally that $r<\beta<\frac{r}{u^*}=\frac{4Kr^2}{( r-1+K)^2}$.
\end{proof}

\begin{lemma}\label{X2} 
 If  \(r>(1+\sqrt{K})^2\) and \(\beta<\frac{r}{2u^*-1}\), then there is only one positive intersection point between \(C_2^2\) and \(C_3^2\).  Furthermore, \(0\le f(x,y)\le 1\)   for \((x,y)\in \Omega_2\) if moreover \(r\le4\). 
\end{lemma}

\begin{proof}
The intersection of the   \(C_2^2\) and \(C_3^2\) are solutions of the second degree polynomial equation
\[
x+\frac{r}{\beta}(1-x)(1+ax)=u^*.
\]
Under the condition \(\beta<\frac{r}{2u^*-1}\), 
\[
\Delta=\left(1-\frac{r}{\beta}-\frac{ar}{\beta}\right)^2-\frac{4ar}{\beta}\left(u^*-\frac{r}{\beta}\right)>0
\]
Hence there are two real roots.  By the Routh test of location of polynomial there is one positive root and one negative root counting the sign change in the first column of the Routh array.
Denote the positive root as \(\bar x\) which is by straightforward calculation
\[
\bar x=\frac{1-\frac{r}{\beta}+\frac{ar}{\beta}+\sqrt{\Delta}}{2\frac{r}{a\beta}}>\frac{1}{2}.
\]

The main issues remained is to show \(f(x,y)\ge0\) as we have shown in the proof of Lemma \ref{f<u^*} \(f(x,y)\le 1\) if \(r\le4\).  The proof is similar to the proof of lemma \ref{f>0}. The only difference is to determine the minimum of the function \(h(x,y)\) on the boundary \(\partial \Omega_2\).  Here we use the facts that  \(h(x,y)\) is a concave function and \(\Omega_2\) is compact, from which we evaluate the minimum on the boundary.  A straightforward computation yields 
\[
\min_{\Omega_2}f(x,y)= r-\beta u^* +\min\{0,\bar x(\beta-r+ra-ra\bar x)\}=r-\beta u^* \bar x(\beta-r+ra-ra\bar x),
\]
where the last equality holds due to \(\beta-r+ra-ra\bar x<0\) by substitution of the expression for \(\bar x\).  A further straightforward calculation gives \(r-\beta u^* \bar x(\beta-r+ra-ra\bar x)>0\), showing \(f(x,y)\ge0\).
\end{proof}

\begin{lemma}\label{X3}
Let \(a=1\).  Assume either that \(u^*>2\) and \(\beta<r\) or that \(\frac54<u^*\le2\) and \(\beta<{r}{v_+}\) where \(v_+=\left( \frac{\sqrt{u^*+1} +\sqrt{u^*-1}}{2}\right)^2>0\).  Then there are two intersection points \(x_1\) and \(x_2\) of the curves 
\(C_2^3\) and \(C_3^3\) satisfying
\(0<x_1<\frac12\) and \(\frac12<x_2<1\), respectively.  Moreover, \(f(x,y)\ge0\) if \((x,y)\in\Omega_3\).
\end{lemma}

\begin{proof}
When \(a=1\) we have the simpler second degree  polynomial equation
\(\frac{r}{\beta} x^2-x+u^*-\frac{r}{\beta}=0\).  Let \(v=\frac{r}{\beta}\).  Now
\(\Delta=4v^2-4u^*v+1=(2v-u^*)^2-((u^*)^2-1)>0\)  if either 
\(v>v_+\) or \(v<v_-\), where \(v_\pm=\frac{u^*\pm\sqrt{(u^*)^2-1}}{2}
=\left( \frac{\sqrt{u^*+1} \pm\sqrt{u^*-1}}{2}\right)^2>0\).  This is equivalent to 
\(\beta<\frac{r}{v_+}\) or \(\beta>\frac{r}{v_-}\). 

Note that \(\Delta<1\) since \(u^*-v>0\).  So the two roots of the polynomial equation
\[
x_1=\frac{1-\sqrt{\Delta}}{2}, x_2=\frac{1+\sqrt{\Delta}}{2}.
\]
satisfy \(0<x_1<\frac12\) and \(\frac12< x_2<1\) if \(\beta<\frac{r}{v_+}\) or \(\beta>\frac{r}{v_-}\). 
Next we estimate the bounds of \(v_-\) and \(v_+\) to get conditions on \(\beta\).  By a basic calculus argument we can find that \(v_+\ge\frac12\) and \(v_-\le\frac12\) for  for \(u^*<1\).  But \(v\) is supposed to be great than \(1\).  Hence \(v_+\ge\frac12\) yields the positive \(\Delta\).  Thus \(u^*>v>\max\{1,v_+\), or equivalently
\[
\frac{r}{u^*}<\beta<r \min\{1,v_+\}=
\begin{cases} 
rv_+ , & \text{ if } \frac54\le u^*<2\\ 
r, &\text{ if } u^*\ge 2
\end{cases}
\]
proving the first part of the lemma.  

The rest of the proof is similar to the proof of the previous lemma.  The evaluation of the minimum of \(h(x,y)\) on \(\Omega_3\) yields
\(\min_{\Omega_3}h(x,y)\ge0\).  Hence \(f(x,y)\ge0\).  
\end{proof}

Note that if we assume that \(r\le4\) then \(f(x,y)\le1\).  

 \begin{proof}[Proof of Proposition \ref{Omega positive}]
(i) Suppose $(S_k,I_k)\in\Omega_1$ and the specified conditions hold. Then by Lemmas \ref{f<u^*} and \ref{f>0} we have $S_{k+1}=f(S_k,I_k)\geq0$. Moreover since $S_k+I_k\leq u^*$, Lemma \ref{u^*} implies that $S_{k+1}+I_{k+1}\leq u^*$.  Finally observe that $g(x,y)=(1-K)y + \frac{\beta xy}{1+ax}\geq0$ for all $x,y>0$.
Thus $(S_{k+1},I_{k+1})\in\Omega_1$ and the result follows by induction.
 
 (ii) By Lemma \ref{X2}, \(f(S_n,I_n)\le 1\) for all \(1\le r\le 4\).   
By Lemma \ref{X2}, \(0\le S_{n+1}\le 1\) if \((1+\sqrt{K})^2<r\le 4\) and \(\beta<\frac{r}{2u^*-1}\). It remains to show that  \(I_{n+1}\le \frac{r}{\beta}(1-S_{n+1})(1+aS_{n+1})\) if \((S_n,I_n)\in \Omega_2\).
If \(S_{n+1},I_{n+1}\) lies in the region to the left of the line \(x=\bar x\) then by Lemma \ref{u^*} \(S_{n+1}+I_{n+1}\le u^*\) implying that \((S_{n+1},I_{n+1})\in\Omega_2\).  Now we assume by contradiction \(I_{n+1}>\frac{r}{\beta}(1-S_{n+1})(1+aS_{n+1})\) when \(\bar x\le S_{n+1}\le 1\).  Hence
\[
 \frac{r}{\beta}(1-S_{n+1})(1+aS_{n+1}< I_{n+1}\le u^*-S_{n+1}\le u^*-\bar x.
\]
Since the function in the left hand side is concave and decreasing, we get
\[
 \frac{r}{\beta}(1-\bar x)(1+a\bar x)<I_{n+1}\le u^*-S_{n+1}\le u^*-\bar x,
\]
contracting to the fact that \(\bar x\) is the coordinator of the intersection point between the line \(x+y=u^*\) and the curve \(y=\frac{r}{\beta}(1- x)(1+ax)\), proving the statement.

(iii) This follows by Lemma \ref{X3} and the similar argument as for (ii).
\end{proof}

\section{Computing the first Lyapunov coefficient \(c\) }\label{compute c}

\subsection{Flip from $E_0$}
To show that the flip bifurcation from $E_0$, happening when $r=3$ and $\beta<\beta_0$, is stable we had to determine the nondegeneracy coefficient 
$$
c = \frac{1}{6}\langle p , C(q,q,q) \rangle - \frac{1}{2}\langle p , B(q,(A-I_2)^{-1}B(q,q)) \rangle,
$$
where $B(x,y), C(x,y,z)$ are given by
\begin{equation}\label{app.Bi}
B_i(x,y) = \sum_{j,k=1}^n \evalat{\diffp{F_i(\xi)}{{\xi_j}{\xi_k}}}{\xi=0}x_jy_k,
\end{equation}
and
\begin{equation}\label{app.Ci}
C_i(x,y) = \sum_{j,k,l=1}^n \evalat{\diffp{F_i(\xi)}{{\xi_j}{\xi_k}{\xi_l}}}{\xi=0}x_jy_kz_l,
\end{equation}
where $i = 1,2$, and $A$ is the Jacobian matrix evaluated at $E_0$.

We have 
\begin{align*}
f(S,I) &= rS(1-S) - \frac{\beta SI}{1+aS}\\
g(S,I) &= (1-K)I + \frac{\beta SI}{1+aS}.
\end{align*}

To shift the fixed point to the origin, define
\begin{align*}
\xi_1 &= S-S_0\\
\xi_2 &=I-I_0 = I
\end{align*}
and note that $\xi_1=\xi_2=0$ if and only if $S = S_0$ and $I = 0$. 

In these new coordinates the system becomes
\begin{equation}\label{syst origin}
\begin{aligned}
\xi_1(n+1) &= f(\xi_1(n) + S_0, \xi_2(n))-S_0\\
\xi_2(n+1) &= g(\xi_1(n) + S_0, \xi_2(n).
\end{aligned}
\end{equation}

We write the system \eqref{syst origin} as
\begin{equation}\label{syst origin linearized}
\begin{pmatrix}
\xi_1(n+1)\\
\xi_2(n+1)
\end{pmatrix}
=
J(E_0)
\begin{pmatrix}
\xi_1(n)\\
\xi_2(n)
\end{pmatrix}
+ 
F(\xi_1(n),\xi_2(n))
\end{equation}
where as usual $J(E_0)$ is the Jacobian matrix evaluated at $E_0$. Then by definition
\begin{equation}\label{F def}
F = 
\begin{pmatrix}
F_1\\
F_2
\end{pmatrix}
=
\begin{pmatrix}
f(\xi_1 + S_0,\xi_2)-S_0\\
g(\xi_1+S_0,\xi_2)
\end{pmatrix}
- J(E_0)
\begin{pmatrix}
\xi_1\\
\xi_2
\end{pmatrix},
\end{equation}
and its Taylor expansion near the origin is given by
$$
F(x) = \frac{1}{2}B(x,x) + \frac{1}{6}C(x,x,x) + O(||x||^4),
$$
with $B(x,x),C(x,x,x)$ given by \eqref{app.Bi} and \eqref{app.Ci}. Our system is two-dimensional, so we have 
$$
B(x,y) =
\begin{pmatrix}
B_1(x,y)\\
B_2(x,y)
\end{pmatrix}.
$$
Form \eqref{F def} we find that
\begin{equation*}
F_1(\xi_1,\xi_2) = r(\xi_1+\frac{r-1}{r})(1-\xi_1+\frac{r-1}{r}) - \frac{\beta(\xi_1+\frac{r-1}{r})\xi_2}{1+a(\xi_1+\frac{r-1}{r})} -\frac{r-1}{r} - (2-r)\xi_1 - \frac{\beta(r-1)\xi_2}{r+a(r-1)},
\end{equation*}
and
\begin{equation*}
F_2(\xi_1,\xi_1) = (1-K)\xi_2+ \frac{\beta(\xi_1+\frac{r-1}{r})\xi_2}{1+a(\xi_1+\frac{r-1}{r})} - (1-K)\xi_2 - \frac{\beta(r-1)\xi_2}{r+a(r-1)}
= \frac{\beta(\xi_1+\frac{r-1}{r})\xi_2}{1+a(\xi_1+\frac{r-1}{r})} - \frac{\beta(r-1)\xi_2}{r+a(r-1)}.
\end{equation*}

Now we can compute partial derivatives. As these computations are completely straight forward but somewhat tedious, we just state that 
$$
\begin{matrix}
\evalat{\diffp[2]{F_1}{\xi_1}}{\xi=0} = -2r, &
\evalat{\diffp{F_1}{{\xi_1}{\xi_2}}}{\xi=0} = -\frac{\beta}{(1+\frac{a(r-1)}{r})^2}, &
\evalat{\diffp[2]{F_1}{\xi_2}}{\xi=0} = 0, \\
\evalat{\diffp[2]{F_2}{\xi_1}}{\xi=0} = 0, &
\evalat{\diffp{F_2}{{\xi_1}{\xi_2}}}{\xi=0} = \frac{\beta}{(1+\frac{a(r-1)}{r})^2}, &
\evalat{\diffp[2]{F_2}{\xi_2}}{\xi=0} = 0.
\end{matrix}
$$
Hence by \eqref{app.Bi} we get
$$
B(x,y) = 
\begin{pmatrix}
-2rx_1y_1 -\frac{\beta}{(1+\frac{a(r-1)}{r})^2}x_1y_2 -\frac{\beta}{(1+\frac{a(r-1)}{r})^2}x_2y_1\\
\frac{\beta}{(1+\frac{a(r-1)}{r})^2}x_1y_2 +\frac{\beta}{(1+\frac{a(r-1)}{r})^2}x_2y_1
\end{pmatrix}.
$$
Since $q = (1~0)^T$ we find that 
$$
B(q,q) = B(1,0,1,0) = 
\begin{pmatrix}
-2r\\
0
\end{pmatrix},
$$
which tells us that 
$$
\evalat{B(q,q)}{r=3} = -6\begin{pmatrix}1\\0\end{pmatrix}.
$$
Finally, the matrix
$$
(A-I_2)^{-1}
=
\left(
\begin{array}{cc}
-2 & -\frac{2 \beta }{2 a+3} \\
0 & \frac{2 \beta }{2 a+3}-K \\
\end{array}
\right)^{-1} =
\left(
\begin{array}{cc}
-\frac{1}{2} & \frac{\beta }{(2 a+3) K-2 \beta } \\
0 & -\frac{2}{2 K-\frac{4 \beta }{2 a+3}} \\
\end{array}
\right)
$$ 
so that
$$
(A-I_2)^{-1}B(q,q) = -6\bigg(-\frac{1}{2}\bigg)\begin{pmatrix}1\\0\end{pmatrix} = 3\begin{pmatrix}1\\0\end{pmatrix},
$$
which implies that 
$$
B(q,(A-I_2)^{-1} B(q,q)) = B(1,0,3,0) = \begin{pmatrix}-6\cdot3\\0\end{pmatrix} = -18\begin{pmatrix}1\\0\end{pmatrix}.
$$
Now we can compute
\begin{equation}\label{the B part}
-\frac{1}{2}\langle p , B(q,(A-I_2)^{-1} B(q,q)) \rangle = -\frac{1}{2} \left\langle \begin{pmatrix} 1\\p_2\end{pmatrix} , -18\begin{pmatrix} 1\\0\end{pmatrix} \right\rangle= 9.
\end{equation}
We are now well on the way. All that remains is to find $C(x,y,z)$ given by \eqref{app.Ci}. Again, the computations are tedious but not very difficult. We just give the results:
$$
\begin{array}{l l}
\evalat{\diffp[3]{F_1}{\xi_1}}{\xi=0} = 0, & 
\evalat{\diffp[3]{F_1}{\xi_2}}{\xi=0} =0 , \\
\evalat{\diffp{F_1}{{\xi_1^2}{\xi_2}}}{\xi=0} = \frac{2a\beta}{(1+\frac{a(r-1)}{r})^3}, &
\evalat{\diffp{F_1}{{\xi_1}{\xi_2^2}}}{\xi=0} = 0,\\
\evalat{\diffp[3]{F_2}{\xi_1}}{\xi=0} = 0, & 
\evalat{\diffp[3]{F_2}{\xi_2}}{\xi=0} =0 , \\
\evalat{\diffp{F_2}{{\xi_1^2}{\xi_2}}}{\xi=0} = -\frac{2a\beta}{(1+\frac{a(r-1)}{r})^3}, &
\evalat{\diffp{F_2}{{\xi_1}{\xi_2^2}}}{\xi=0} = 0.
\end{array}
$$
Using this and \eqref{app.Ci} we get
$$
C(x,y,z) = 
\begin{pmatrix}
\frac{2a\beta}{(1+\frac{a(r-1)}{r})^3}x_1y_1z_2 + \frac{2a\beta}{(1+\frac{a(r-1)}{r})^3}x_1y_2z_1 + \frac{2a\beta}{(1+\frac{a(r-1)}{r})^3}x_2y_1z_1\\
-\frac{2a\beta}{(1+\frac{a(r-1)}{r})^3}x_1y_1z_2 - \frac{2a\beta}{(1+\frac{a(r-1)}{r})^3}x_1y_2z_1 - \frac{2a\beta}{(1+\frac{a(r-1)}{r})^3}x_2y_1z_1
\end{pmatrix},
$$
and we see that
$$
C(q,q,q) = C(1,0,1,0,1,0) = \begin{pmatrix}0\\0\end{pmatrix}
$$
which entails
\begin{equation}\label{the C part}
\frac{1}{6}\langle p , C(q,q,q) \rangle = 0.
\end{equation}

Now, using \eqref{the B part} and \eqref{the C part} we finally get
$$
c = \frac{1}{6}\langle p , C(q,q,q) \rangle - \frac{1}{2}\langle p , B(q,(A-I_2)^{-1}B(q,q)) \rangle = 0 + 9 = 9.
$$

\subsection{Flip from $E_1$}

Again, our aim is to compute 
$$
c = \frac{1}{6}\langle p , C(q,q,q) \rangle - \frac{1}{2}\langle p , B(q,(A-I_2)^{-1}B(q,q)) \rangle.
$$
Again, we shift the fixed point to the origin by defining
\begin{align*}
\xi_1 &= S-S_1 = S - \frac{K}{\beta-aK}\\
\xi_2 &=I-I_1= I -\frac{r-1}{\beta-aK} -\frac{rK}{(\beta-aK)^2}.
\end{align*}
Then $\xi_1=\xi_2=0$ if and only if $S = S_1$ and $I = I_1$. Again, we write
\begin{equation}\label{syst origin linearized E_1}
\begin{pmatrix}
\xi_1(n+1)\\
\xi_2(n+1)
\end{pmatrix}
=
J(E_1)
\begin{pmatrix}
\xi_1(n)\\
\xi_2(n)
\end{pmatrix}
+ 
F(\xi_1(n),\xi_2(n)),
\end{equation}
so that, again
\begin{equation}\label{F def E_1}
F = 
\begin{pmatrix}
F_1\\
F_2
\end{pmatrix}
=
\begin{pmatrix}
f(\xi_1 + S_1,\xi_2+I_1)-S_1\\
g(\xi_1+S_1,\xi_2+I_1)-I_1
\end{pmatrix}
- J(E_1)
\begin{pmatrix}
\xi_1\\
\xi_2
\end{pmatrix},
\end{equation}
and its Taylor expansion near the origin is given by
$$
F(x) = \frac{1}{2}B(x,x) + \frac{1}{6}C(x,x,x) + O(||x||^4),
$$
with $B(x,x),C(x,x,x)$ given by \eqref{app.Bi} and \eqref{app.Ci}. Our system is two-dimensional, so we have 
$$
B(x,y) =
\begin{pmatrix}
B_1(x,y)\\
B_2(x,y)
\end{pmatrix}.
$$
We see that
$$
F_1(\xi_1,\xi_2) = r(\xi_1+S_1)(1-\xi_1-S_1) - \frac{\beta(\xi_2+I_1)(\xi_1+S_1)}{1+a(\xi_1+S_1)} - S_1 - a_{11}\xi_1 + K\xi_2
$$
and
$$ 
F_2(\xi_1,\xi_2) = (1-K)(\xi_2 + I_1) + \frac{\beta(\xi_2+I_1)(\xi_1+S_1)}{1+a(\xi_1+S_1)} - I_1 -a_{21}\xi_1 - \xi_2.
$$
Again, the computation of partial derivatives is not particularly interesting, so we just state that
$$
\begin{matrix}
\evalat{\diffp[2]{F_1}{\xi_1}}{\xi=0} = \frac{2 a (a K-\beta ) (a K (r-1)+\beta +K r-\beta r)}{\beta ^2}-2 r, \\
\evalat{\diffp{F_1}{{\xi_1}{\xi_2}}}{\xi=0} =-\frac{(\beta -a K)^2}{\beta },~
\evalat{\diffp[2]{F_1}{\xi_2}}{\xi=0} = 0, \\
\evalat{\diffp[2]{F_2}{\xi_1}}{\xi=0} = -\frac{2 a (a K-\beta ) (a K (r-1)+\beta +K r-\beta r)}{\beta ^2}, \\
\evalat{\diffp{F_2}{{\xi_1}{\xi_2}}}{\xi=0} =\frac{(\beta -a K)^2}{\beta }, ~
\evalat{\diffp[2]{F_2}{\xi_2}}{\xi=0} = 0,
\end{matrix}
$$
which means that
$$
B(x,y) = 
\begin{pmatrix}
x_1 y_1 \left(\frac{2 a (a K-\beta ) (a K (r-1)+\beta +K r-\beta r)}{\beta ^2}-2 r\right)-\frac{x_2 y_1 (\beta -a K)^2}{\beta }-\frac{x_1 y_2 (\beta -a K)^2}{\beta }\\
\frac{x_1 (a K-\beta ) \left(\beta y_2 (a K-\beta )-2 a y_1 (K (a (r-1)+r)+\beta -\beta r)\right)+\beta x_2 y_1 (\beta -a K)^2}{\beta ^2}
\end{pmatrix}.
$$
Next, we compute
$$
\begin{array}{l l}
\evalat{\diffp[3]{F_1}{\xi_1}}{\xi=0} =\tfrac{6 a^2 (\beta -a K)^2 (a K (r-1)+\beta +K r-\beta r)}{\beta ^3} , & 
\evalat{\diffp[3]{F_1}{\xi_2}}{\xi=0} =-\tfrac{2 a (a K-\beta )^3}{\beta ^2}, \\
\evalat{\diffp{F_1}{{\xi_1^2}{\xi_2}}}{\xi=0} = 0, &
\evalat{\diffp{F_1}{{\xi_1}{\xi_2^2}}}{\xi=0} = 0,\\
\evalat{\diffp[3]{F_2}{\xi_1}}{\xi=0} = -\tfrac{6 a^2 (\beta -a K)^2 (a K (r-1)+\beta +K r-\beta r)}{\beta ^3}, & 
\evalat{\diffp[3]{F_2}{\xi_2}}{\xi=0} = \tfrac{2 a (a K-\beta )^3}{\beta ^2}, \\
\evalat{\diffp{F_2}{{\xi_1^2}{\xi_2}}}{\xi=0} =0, &
\evalat{\diffp{F_2}{{\xi_1}{\xi_2^2}}}{\xi=0} = 0,
\end{array}
$$
which allows us to determine
$$
C(x,y,z) = 
\begin{pmatrix}
C_1(x,y,z)\\
C_2(x,y,z)
\end{pmatrix},
$$
where 
\begin{align*}
C_1(x,y,z) =& \tfrac{6 a^2 x_1 y_1 z_1 (\beta -a K)^2 (a K (r-1)+\beta +K r-\beta r)}{\beta ^3}-\tfrac{2 a x_2 y_1 z_1 (a K-\beta )^3}{\beta ^2}\\
&-\tfrac{2 a x_1 y_2 z_1 (a K-\beta )^3}{\beta ^2}-\tfrac{2 a x_1 y_1 z_2 (a K-\beta )^3}{\beta ^2}
\end{align*}
and
\begin{align*}
C_2(x,y,z) =& -\tfrac{6 a^2 x_1 y_1 z_1 (\beta -a K)^2 (a K (r-1)+\beta +K r-\beta r)}{\beta ^3}+\tfrac{2 a x_2 y_1 z_1 (a K-\beta )^3}{\beta ^2}\\
&+\tfrac{2 a x_1 y_2 z_1 (a K-\beta )^3}{\beta ^2}+\tfrac{2 a x_1 y_1 z_2 (a K-\beta )^3}{\beta ^2}.
\end{align*}
This then would in principle allow us to compute 
$$
c = \frac{1}{6}\langle p , C(q,q,q) \rangle - \frac{1}{2}\langle p , B(q,(A-I_2)^{-1}B(q,q)) \rangle,
$$
where we would have to replace $\beta$ by $\beta_1$ everywhere. Unfortunately, even using Mathematica this is a very complicated expression. Numerical computations show that $c$ can be both positive and negative, which means by continuity and the intermediate value theorem that it can also be zero. 

\section{Computing the first Lyapunov coefficient \(d\)}\label{computing d}

We give briefly the steps one goes through to compute the nondegeneracy coefficient $d$. In appendix \ref{compute c} we have computed the multilinear functions $B(x,y)$ and $C(x,y,z)$ for $E_1$. They remain the same here. First, we note that the characteristic polynomial is
$$
P(z) = z^2-(a_{11}+1)z + a_{11}+Ka_{21},
$$
which yields the eigenvalues (that we know are complex)
$$
\mu_{1,2} = \frac{a_{11}+1\pm i\sqrt{4(a_{11}+Ka_{21})-(a_{11}+1)^2}}{2},
$$
and we discussed before that $\mu_{1,2} = e^{\pm i\theta_0} = \sigma \pm \omega$ where $2\sigma = a_{11} + 1$. It follows from Euler's formula that 
$\sigma = \cos{\theta_0}$, and hence $\theta_0 = \arccos(\frac{a_{11}+1}{2})$.

Now, we wish to determine a generalized eigenvector $q$ of $A$. Such a vector satisfies
$$
Aq=e^{i\theta_0}q,~A\bar{q}=e^{-i\theta_0}\bar{q}.
$$
We get $q$ by solving 
$$
\begin{pmatrix}
a_{11}-e^{i\theta_0} & -K \\
a_{21} & 1-e^{i\theta_0}
\end{pmatrix}
\begin{pmatrix}
q_1\\
q_2
\end{pmatrix}
= 
\begin{pmatrix}
0\\
0
\end{pmatrix}.
$$
We may choose $q_1 = 1$ which yields $q_2 = \frac{a_{11}-e^{i\theta_0}}{K}$. Hence 
$$
q =
\begin{pmatrix}
1\\
\frac{a_{11}-e^{i\theta_0}}{K}
\end{pmatrix}.
$$
Next, we seek a generalized adjoint eigenvector $p$, which we normalize as before. Then $p$ must satisfy
$$
A^Tp=e^{i\theta_0}p,~A^T\bar{p}=e^{-i\theta_0}\bar{p},~\langle p , q \rangle = 1,
$$
which gives us three equations to solve:
$$
\begin{cases}
p_1+\frac{a_{11}-e^{i\theta_0}}{K}p_2 = 1\\
p_1(a_{11}-e^{i\theta_0})+a_{21}p_2=0\\
-Kp_1 + (1-e^{i\theta_0})p_2=0.
\end{cases}
$$
This yields
$$
p =\frac{1}{a_{11} -2e^{i\theta_0}+1} 
\begin{pmatrix}
1-e^{i\theta_0}\\
K
\end{pmatrix}.
$$
Now, using Mathematica, replacing $\beta$ everywhere by $\beta_2$, we can compute
\begin{multline*}
d=\frac{1}{2}Re\bigg(e^{-i\theta_0}\bigg[\langle p , C(q,q,\bar{q}) \rangle + 2\langle p , B(q,(A-I_n)^{-1}B(q,\bar{q})\rangle\\
+ \langle p , B(\bar{q} , (e^{2i\theta_0}I_n-A)^{-1}B(q,q))\rangle\bigg]\bigg).
\end{multline*}
Unfortunately, this is a massively complicated expression, so we have to resort to numerical experimentation. This strongly suggests that $d<0$ for all choices of $a$ and $K$ when $1<r<r_{max}$. Further, as $r$ approaches 1 from above, it seems very clear that $d\to-\infty$. If one plots $d$ as a function of $r$, it reaches a local maximum for $r$ between 1 and 3. Usually this maximum is attained quite close to $r = 1$. All this strongly suggests that $d<0$ for $1<r<r_{max}$. A graph is shown in figure \ref{d_LimitCycle}.

\end{document}